\documentclass[11pt,twoside,a4paper]{article}
\usepackage{amsmath}
\usepackage{amssymb}
\usepackage[latin1]{inputenc}
\usepackage[T1]{fontenc}   
\usepackage[english]{babel}
\title{Examples of Com-PreLie Hopf algebras}
\date{}
\author{Lo\"\i c Foissy\\ \\
{\small \it Fédération de Recherche Mathématique du Nord Pas de Calais FR 2956}\\
{\small \it Laboratoire de Mathématiques Pures et Appliquées Joseph Liouville}\\
{\small \it Université du Littoral Côte dOpale-Centre Universitaire de la Mi-Voix}\\ 
{\small \it 50, rue Ferdinand Buisson, CS 80699,  62228 Calais Cedex, France}\\ \\
{\small \it email: foissy@lmpa.univ-littoral.fr}}

\newtheorem{defi}{\indent Definition}
\newtheorem{lemma}[defi]{\indent Lemma}
\newtheorem{cor}[defi]{\indent Corollary}
\newtheorem{theo}[defi]{\indent Theorem}
\newtheorem{prop}[defi]{\indent Proposition}

\newenvironment{proof}{{\bf Proof.}}{\hfill $\Box$}

\def\shuff#1#2{\mathbin{
      \hbox{\vbox{\hbox{\vrule \hskip#2 \vrule height#1 width 0pt}\hrule}\vbox{\hbox{\vrule \hskip#2 \vrule height#1 width 0pt\vrule }\hrule}}}}
\def\shuffl{{\mathchoice{\shuff{7pt}{3.5pt}}{\shuff{6pt}{3pt}}{\shuff{4pt}{2pt}}{\shuff{3pt}{1.5pt}}}}
\def\shuffle{\, \shuffl \,}

\newcommand{\K}{\mathbb{K}}
\newcommand{\tdelta}{\tilde{\Delta}}
\newcommand{\g}{\mathfrak{g}}
\newcommand{\N}{\mathbb{N}}
\newcommand{\bfDelta}{\mathbf{\Delta}}
\newcommand{\bftDelta}{\mathbf{\tdelta}}

\begin{document}

\maketitle

ABSTRACT. We gives examples of Com-PreLie bialgebras, that is to say bialgebras with a preLie product satisfying certain compatibilities.
Three families are defined on shuffle algebras: one associated to linear endomorphisms, one associated to linear form, one associated to preLie algebras.
We also give all graded preLie product on $\K[X]$, making this bialgebra a Com-PreLie bialgebra, and classify all connected cocommutative
Com-PreLie bialgebras.\\


AMS CLASSIFICATION. 17D25\\

\tableofcontents

\section*{Introduction}

The composition of Fliess operators \cite{Gray} gives a group structure on set of noncommutative formal series $\K\langle\langle x_0,x_1\rangle\rangle$
in two variables $x_0$ and $x_1$. For example, let us consider the following formal series:
\begin{align*}
A&=a_\emptyset+a_0 x_0+a_1x_1+a_{00}x_0^2+a_{01}x_0x_1+a_{10}x_1x_0+a_{11}x_1^2+\ldots,\\
B&=b_\emptyset+b_0 x_0+b_1x_1+b_{00}x_0^2+b_{01}x_0x_1+b_{10}x_1x_0+b_{11}x_1^2+\ldots,\\
B&=c_\emptyset+c_0 x_0+c_1x_1+c_{00}x_0^2+c_{01}x_0x_1+c_{10}x_1x_0+c_{11}x_1^2+\ldots;
\end{align*}
if $C=A.B$, then:
\begin{align*}
c_\emptyset&=a_\emptyset+b_\emptyset,\\
c_0&=a_0+b_0+a_1b_\emptyset,\\
c_{00}&=a_{00}+b_{00}+a_{01}b_\emptyset+a_{10}b_\emptyset+a_{11}b_\emptyset^2+a_1b_0,\\
c_{01}&=a_{01}+b_{01}+a_{11}b_\emptyset+a_1b_1,\\
c_{10}&=a_{10}+b_{10}+a_{11}b_\emptyset,\\
c_{11}&=a_{11}+b_{11}.
\end{align*}
This quite complicated structure can be more easily described with the help of the Hopf algebra of coordinates of this group;
this leads to a Lie algebra structure on the algebra $\K\langle x_0,x_1\rangle$ of noncommutative polynomials in two variables,
which is in a certain sense the infinitesimal structure associated to the group of Fliess operators.
As explained in \cite{Foissyprelie}, this Lie bracket comes from a nonassociative, preLie product $\bullet$. For example:
\begin{align*}
x_0x_0\bullet x_0&=0,&x_0x_0\bullet x_1&=0,\\
x_0x_1\bullet x_0&=x_0x_0x_0,&x_0x_1\bullet x_1&=x_0x_0x_1,\\
x_1x_0\bullet x_0&=2x_0x_0x_0,&x_1x_0\bullet x_1&=x_0x_0x_1+x_0x_1x_0,\\
x_1x_1\bullet x_0&=x_1x_0x_0+x_0x_1x_0+x_0x_0x_1,&x_1x_1\bullet x_1&=x_1x_0x_1+2x_0x_1x_1.
\end{align*}
Moreover, $\K\langle x_0,x_1\rangle$ is naturally a Hopf algebra with the shuffle product $\shuffle$ and the deconcatenation coproduct $\Delta$,
and it turns out that there exists compatibilities between this Hopf-algebraic structure and the preLie product $\bullet$:
\begin{itemize}
\item For all $a,b,c \in A$, $(a\shuffle b)\bullet c=(a\bullet c)\shuffle b+a\shuffle (b\bullet c)$.
\item For all $a,b\in A$, $\Delta(a\bullet b)=a^{(1)}\otimes a^{(2)}\bullet b+a^{(1)}\bullet b^{(1)}\otimes a^{(2)}\shuffle b^{(2)}$,
with Sweedler's notation.
\end{itemize}
this is a Com-PreLie bialgebra (definition \ref{1}). Moreover, the shuffle bracket can be induced by the half-shuffle product $\prec$,
and there is also a compatibility between $\prec$ and $\bullet$: 
\begin{itemize}
\item For all $a,b,c \in A$, $(a\prec b)\bullet c=(a\bullet c)\prec b+a\prec (b\bullet c)$.
\end{itemize}
we obtain a Zinbiel-PreLie bialgebra.\\

Our aim in the present text is to give examples of other Com-PreLie algebras or bialgebras.
We first introduce three families, all based on the shuffle Hopf algebra $T(V)$ associated to a vector space $V$.
\begin{enumerate}
\item The first family $T(V,f)$, introduced in \cite{Foissyprelie2}, is parametrized by linear endomorphism of $V$. 
For example, if $x_1,x_2,x_3\in V$, $w\in T(V)$:
\begin{align*}
x_1 \bullet w&=f(x_1) w,\\
x_1x_2\bullet w&=x_1 f(x_2)w+f(x_1)(x_2 \shuffle w),\\
x_1x_2x_3 \bullet w&=x_1x_2f(x_3)w+x_1f(x_2)(x_3 \shuffle w)+f(x_1)(x_2x_3\shuffle w).
\end{align*}
In particular, if $V=Vect(x_0,x_1)$, $f(x_0)=0$ and $f(x_1)=x_0$, we recover in this way the Com-PreLie bialgebra of Fliess operators.
\item The second family $T(V,f,\lambda)$ is indexed by pairs $(f,\lambda)$, where $f$ is a linear form on $V$ and $\lambda$ is a scalar.
For example, if $x,y_1,y_2,y_3 \in V$ and $w\in T(V)$:
\begin{align*}
xw\bullet y_1&=f(x) w\shuffle y_1,\\
xw\bullet y_1y_2&=f(x)(w\shuffle y_1y_2+\lambda f(y_1) w\shuffle y_2),\\
xw\bullet y_1y_2y_3&=f(x)(w\shuffle y_1y_2y_3+\lambda f(y_1)w\shuffle y_2y_3+\lambda^2f(y_1)f(y_2)w\shuffle y_3).
\end{align*}
We obtain a Com-PreLie algebra, but generally not a Com-PreLie bialgebra. Nevertheless, the subalgebra $coS(V)$ generated by $V$
is a Com-PreLie bialgebra. Up to an isomorphism, the symmetric algebra becomes a Com-PreLie bialgebra, denoted by $S(V,f,\lambda)$.
\item If $\star$ is a preLie product on $V$, then it can be extended in a product on $T(V)$, making it a Com-PreLie bialgebra denoted by $T(V,\star)$.
For example, if $x_1,x_2,x_3,y\in V$, $w\in T(V)$.
\begin{align*}
x_1\bullet yw&=(x_1\star y)w,\\
x_1x_2\bullet yw&=(x_1\star y)(x_2\shuffle w)+x_1(x_2\star y) w,\\
x_1x_2x_3\bullet yw&=(x_1\star y)(x_2x_3\shuffle w)+x_1(x_2\star y)(x_3\shuffle w)+x_1x_2(x_3\star y)w.
\end{align*} \end{enumerate}
These examples answer some questions on Com-PreLie bialgebras. According to proposition \ref{4}, if $A$ is a Com-PreLie bialgebra,
the map $f_A$ defined by $f_A(x)=x\bullet 1_A$ is an endomorphism of $Prim(A)$; if $f_A=0$, then $Prim(A)$ is a PreLie subalgebra
of $A$. Then:
\begin{itemize}
\item If $A=T(V,f)$, then $f_A=f$, which proves that any linear endomorphim can be obtained in this way.
\item If $A=T(V,\star)$, then $f_A=0$ and the preLie product on $Prim(A)$ is $\star$, which proves that any preLie product can be obtained in this way.
\end{itemize}

The next section is devoted to the algebra $\K[X]$. We first classify preLie products making it a graded Com-PreLie algebra:
this gives four families of Com-PreLie algebras described in theorem \ref{18}, including certain cases of $T(V,f)$. Only a few of them are
compatible with the coproduct of $\K[X]$ (proposition \ref{23}). The last paragraph gives a classification of all connected, cocommutative
Com-PreLie bialgebras (theorem \ref{24}): up to an isomorphism these are the $S(V,f,\lambda)$ and examples on $\K[X]$. \\

{\bf Aknowledgment.} The research leading these results was partially supported by the French National Research Agency under the reference
ANR-12-BS01-0017.\\

{\bf Notations.} \begin{enumerate}
\item $\K$ is a commutative field of characteristic zero. All the objects (vector spaces, algebras, coalgebras, preLie algebras$\ldots$)
in this text will be taken over $\K$.
\item Let $A$ be a bialgebra. 
\begin{enumerate}
\item We shall use Swwedler's notation $\Delta(a)=a^{(1)}\otimes a^{(2)}$ for all $a\in A$.
\item We denote by $A_+$ the augmentation ideal of $A$, and by $\tdelta$ the coassociative coproduct defined by:
$$\tdelta:\left\{\begin{array}{rcl}
A_+&\longrightarrow& A_+\otimes A_+\\
a&\longrightarrow& \Delta(a)-a\otimes 1_A-1_A\otimes a.
\end{array}\right.$$
We shall use Sweedler's notation $\tdelta(a)=a'\otimes a''$ for all $a\in A_+$.
\end{enumerate} \end{enumerate}

 \section{Com-PreLie and Zinbiel-PreLie algebras}

\subsection{Definitions}

\begin{defi}\label{1} \begin{enumerate}
\item A \emph{Com-PreLie algebra} \cite{Mansuy}  is a family $A=(A,\shuffle,\bullet)$, where $A$ is a vector space and 
$\shuffle$ and $\bullet$ are bilinear products on $A$, such that:
\begin{enumerate}
\item $(A,\shuffle)$ is an associative, commutative algebra.
\item $(A,\bullet)$ is a (right) preLie algebra, that is to say, for all $a,b,c\in A$:
$$(a\bullet b)\bullet c-a\bullet(b\bullet c)=(a\bullet c)\bullet b-a\bullet(c\bullet b).$$
\item For all $a,b,c \in A$, $(a\shuffle b)\bullet c=(a\bullet c)\shuffle b+a\shuffle (b\bullet c)$.
\end{enumerate}
\item A \emph{Com-PreLie bialgebra} is a family $(A,\shuffle,\bullet,\Delta)$, such that:
\begin{enumerate}
\item $(A,\shuffle,\bullet)$ is a unitary Com-PreLie algebra.
\item $(A,\shuffle,\Delta)$ is a bialgebra.
\item For all $a,b\in A$, $\Delta(a\bullet b)=a^{(1)}\otimes a^{(2)}\bullet b+a^{(1)}\bullet b^{(1)}\otimes a^{(2)}\shuffle b^{(2)}$.
\end{enumerate}
We shall say that $A$ is \emph{unitary} if the associative algebra $(A,\shuffle)$ has a unit.
\item A \emph{Zinbiel-PreLie algebra}  is a family $A=(A,\prec,\bullet)$, where $A$ is a vector space and 
$\prec$ and $\bullet$ are bilinear products on $A$, such that:
\begin{enumerate}
\item $(A,\prec)$ is a Zinbiel algebra (or shuffle algebra, \cite{Schutz,Loday2,FoissyPatras}) that is to say, for all $a,b,c \in A$:
$$(a\prec b) \prec c=a\prec (b\prec c+c\prec b).$$
\item $(A,\bullet)$ is a preLie algebra.
\item For all $a,b,c \in A$, $(a\prec b)\bullet c=(a\bullet c)\prec b+a\prec (b\bullet c)$.
\end{enumerate}
\item A \emph{Zinbiel-PreLie bialgebra} is a family $(A,\shuffle,\prec,\bullet,\Delta)$ such that:
\begin{enumerate}
\item $(A,\shuffle,\bullet,\Delta)$ is a Com-PreLie bialgebra. 
\item $(A_+,\prec,\bullet)$ is a Zinbiel-PreLie algebra, and for all $x,y \in A_+$, $x\prec y+y\prec x=x\shuffle y$.
\item For all $a,b\in A_+$:
$$\tdelta(a\prec b)=a'\prec b'\otimes a''\shuffle b''+a'\prec b\otimes a''+a'\otimes a''\shuffle b+a\prec b'\otimes b''+a\otimes b.$$
\end{enumerate}\end{enumerate}\end{defi}

{\bf Remarks.} \begin{enumerate}
\item If $(A,\shuffle,\bullet,\Delta)$ is a Com-PreLie bialgebra, then for any $\lambda \in \K$, $(A,\shuffle,\lambda \bullet,\Delta)$ also is. 
\item If $A$ is a Zinbiel-preLie algebra, then the product $\shuffle$ defined by $a\shuffle b=a\prec b+b\prec a$ is associative and commutative,
and $(A,\shuffle,\bullet)$ is a Com-PreLie algebra. Moreover, if $A$ is a Zinbiel-PreLie bialgebra, it is also a Com-PreLie bialgebra.
\item If $A$ is a Zinbiel-PreLie bialgebra, the product $\shuffle$ is entirely determined by $\prec$:
we can omit $\shuffle$ in the description of a Zinbiel-PreLie bialgebra.
\item If $A$ is a Zinbiel-PreLie bialgebra, we extend $\prec$ by $a\prec 1_A=a$ and $1_A\prec a=0$ for all $a\in A_+$. 
Note that $1_A\prec 1_A$ is not defined. 
\item If $A$ is a Com-Prelie bialgebra, if $a,b\in A_+$:
\begin{align*}
\tdelta(a\bullet 1_A)&=a'\otimes a''\bullet 1_A+a'\bullet 1_A\otimes a'',\\
\tdelta(a\bullet b)&=a'\otimes a''\bullet b+a\bullet 1_A \otimes b+a\bullet b'\otimes b''\\
&+a'\bullet 1_A\otimes a''\shuffle b+a'\bullet b\otimes a''+a'\bullet b'\otimes a''\shuffle b'',
\end{align*}
as we shall prove later (lemma \ref{3}) that $1_A\bullet c=0$ for all $c\in A$.
\end{enumerate}

Associative algebras are preLie. However, Com-PreLie algebras are rarely associative:

\begin{prop}
Let $A=(A,\shuffle,\bullet)$ be a Com-PreLie algebra, such that for all $x\in A,$ $x\shuffle x=0$ if, and only if, $x=0$. 
If $\bullet$ is associative, then it is zero.
\end{prop}

\begin{proof} Let $x,y\in A$. 
\begin{align*}
((x\shuffle x)\bullet y) \bullet y&=2((x\bullet y)\shuffle x)\bullet y\\
&=2((x\bullet y)\bullet y)\shuffle x+2(x\bullet y)\shuffle (x\bullet y)\\
&=2(x\bullet (y \bullet y))\shuffle x+2(x\bullet y)\shuffle (x\bullet y)\\
&=(x\shuffle x)\bullet (y \bullet y)+2(x\bullet y)\shuffle (x\bullet y).
\end{align*}
Hence, $(x\bullet y)\shuffle (x\bullet y)=0$. As $A$ is a domain, $x\bullet y=0$. \end{proof}\\

Hence, in our examples below, which are integral domains (shuffle algebras or symmetric algebras),
the preLie product is associative if, and only if, it is zero. Here is another example, where $\bullet$ is associative.
We take $A=Vect(1,x)$, with the products defined by:
\begin{align*}
&\begin{array}{c|c|c}
\shuffle&1&x\\
\hline 1&1&x\\
\hline x&x&0
\end{array}&&\begin{array}{c|c|c}
\bullet &1&x\\
\hline 1&0&0\\
\hline x&0&x
\end{array}\end{align*}
If the characteristic of the base field $\K$ is $2$, this is a Com-PreLie bialgebra, with the coproduct defined by $\Delta(x)=x\otimes 1+1\otimes x$.

\subsection{Linear endomorphism on primitive elements}

\begin{lemma}\label{3} \begin{enumerate}
\item Let $A$ be a Com-PreLie algebra. For all $a \in A$, $1_A\bullet a=0$.
\item  Let $A$ be a Com-PreLie bialgebra, with counit $\varepsilon$. For all $a,b \in A$, $\varepsilon(a\bullet b)=0$.
\end{enumerate}\end{lemma}

\begin{proof} 1. Indeed, $1_A\bullet a=(1_A.1_A)\bullet a=(1_A\bullet a).1_A+1_A.(1_A\bullet a)=2(1_A\bullet a)$, so $1_A\bullet a=0$.\\

2. For all $a,b \in A$:
\begin{align*}
\varepsilon(a\bullet b)&=(\varepsilon \otimes \varepsilon)\circ \Delta(a\bullet b)\\
&=\varepsilon(a^{(1)})\varepsilon( a^{(2)}\bullet b)+\varepsilon(a^{(1)}\bullet b^{(1)})\varepsilon(a^{(2)}
\shuffle b^{(2)})\\
&=\varepsilon(a^{(1)})\varepsilon( a^{(2)}\bullet b)
+\varepsilon(a^{(1)}\bullet b^{(1)})\varepsilon(a^{(2)})\varepsilon(b^{(2)})\\
&=\varepsilon(a\bullet b)+\varepsilon(a\bullet b),
\end{align*}
so $\varepsilon(a\bullet b)=0$. \end{proof}\\

{\bf Remark.} Consequently, if $a$ is primitive:
$$\Delta(a\bullet b)=1_A\otimes a\bullet b+a\bullet b^{(1)}\otimes b^{(2)}.$$
So the map $b\longrightarrow a\bullet b$ is a $1$-cocycle for the Cartier-Quillen cohomology \cite{Connes}.\\

If $A$ is a Com-PreLie bialgebra, we denote by $Prim(A)$ the space of its primitive elements:
$$Prim(A)=\{a\in A\mid \Delta(a)=a\otimes 1+1\otimes a\}.$$
We define an endomorphism of $Prim(A)$ in the following way:

\begin{prop}\label{4}
Let $A$ be a Com-PreLie bialgebra. 
\begin{enumerate}
\item If $x \in Prim(A)$, then $x\bullet 1_A\in Prim(A)$.
We denote by $f_A$ the map:
$$f_A:\left\{\begin{array}{rcl}
Prim(A)&\longrightarrow&Prim(A)\\
a&\longrightarrow&a\bullet 1_A.
\end{array}\right.$$
\item If $f_A=0$, then $Prim(A)$ is a preLie subalgebra of $A$.
\end{enumerate}\end{prop}

\begin{proof} 1. Indeed, if $a$ is primitive:
\begin{align*}
\Delta(a\bullet 1_A)&=a\otimes 1_A\bullet 1_A+1_A\otimes a\bullet 1_A+a\bullet 1_A\otimes 1_A\shuffle 1_A+1_A\bullet 1_A\otimes a\shuffle 1_A\\
&=0+1_A\otimes 1_A\bullet a+a\bullet 1_A\otimes 1_A+0,
\end{align*}
so $a\bullet 1_A$ is primitive. \\

2. Let $a,b \in Prim(A)$.
\begin{align*}
\Delta(a\bullet b)&=a\otimes 1_A \bullet b+1_A\otimes a\bullet b+1_A\bullet 1_A\otimes a\shuffle b
+a\bullet 1_A\otimes b+1_A\bullet b\otimes a+a\bullet b\otimes 1_A\\
&=1_A\otimes a\bullet b+a\bullet b\otimes 1_A. 
\end{align*}
So $a\bullet b \in Prim(A)$. \end{proof}

\section{Examples on shuffle algebras}

 Let $V$ be a vector space and let $f:V\longrightarrow V$ be any linear map.
The tensor algebra $T(V)$ is given the shuffle product $\shuffle$, the half-shuffle $\prec$ and the deconcatenation coproduct $\Delta$, making it a bialgebra.
Recall that these products can be inductively defined in the following way: if $x,y \in V$, $u,v \in T(V)$:
\begin{align*}
&\left\{\begin{array}{rcl}
1\prec yv&=&0,\\
xu\prec v&=&x(u\prec v+v\prec u),\\
\end{array}\right.&\left\{\begin{array}{rcl}
1 \shuffle v&=&0,\\
xu\shuffle yv&=&x(u\shuffle yv)+y(xu\shuffle v).
\end{array}\right. \end{align*}
For any $x_1,\ldots,x_n \in V$:
$$\Delta(x_1\ldots x_n)=\sum_{i=0}^n x_1\ldots x_i\otimes x_{i+1}\ldots x_n.$$
For all linear map $F:V\longrightarrow W$, we define the map:
$$T(F):\left\{\begin{array}{rcl}
T(V)&\longrightarrow&T(W)\\
x_1\ldots x_n&\longrightarrow&F(x_1)\ldots F(x_n).
\end{array}\right.$$
This a Hopf algebra morphism from $T(V)$ to $T(W)$. \\

The subalgebra of $(T(V),\shuffle)$ generated by $V$ is denoted by $coS(V)$.
It is the largest cocommutative Hopf subalgebra of $(T(V),\shuffle,\Delta)$; it is generated by the symmetric tensors of elements of $V$.

\subsection{Com-PreLie algebra attached to a linear endomorphism}

We described in \cite{Foissyprelie2} a first family of Zinbiel-PreLie bialgebras;
coming from a problem of composition of Fliess operators in Control Theory. Let $f$ be an endomorphism of a vector space $V$.
We define a bilinear product $\bullet$ on $T(V)$ inductively on the length of words in the following way: if $x \in V$, $v,w\in T(V)$,
\begin{align*}
1 \bullet w&=0,&xv\bullet w&=x(v\bullet w)+f(x)(v \shuffle w).
\end{align*}
Then $(T(V),\prec,\bullet,\Delta)$ is a Zinbiel-PreLie bialgebra, denoted by $T(V,f)$. Moreover, $f_{T(V,f)}=f$. \\

{\bf Examples.} If $x_1,x_2,x_3\in V$, $w\in T(V)$:
\begin{align*}
x_1 \bullet w&=f(x_1) w,\\
x_1x_2\bullet w&=x_1 f(x_2)w+f(x_1)(x_2 \shuffle w),\\
x_1x_2x_3 \bullet w&=x_1x_2f(x_3)w+x_1f(x_2)(x_3 \shuffle w)+f(x_1)(x_2x_3\shuffle w).
\end{align*}
More generally, if $x_1,\ldots,x_n \in V$ and $w \in T(V)$:
$$x_1\ldots x_n \bullet w=\sum_{i=1}^n x_1\ldots x_{i-1}f(x_i)(x_{i+1}\ldots x_n \shuffle w).$$

This construction is functorial: let $V$ and $W$ be two vector spaces, $f$ an endomorphism of $V$ and $g$ an endomorphism of $W$;
let $F:V\longrightarrow W$, such that $g\circ F=F\circ f$. Then $T(F)$ is a morphism of Zinbiel-PreLie bialgebras from $T(V,f)$ to $T(W,g)$.

\begin{prop}
Let $\blacklozenge$ be a preLie product on $(T(V),\shuffle,\Delta)$, making it a Com-PreLie bialgebra, such that for all $k,l\in \N$,
$V^{\otimes k}\\blacklozenge V^{\otimes l}\subseteq V^{\otimes (k+l)}$. There exists a $f\in End(V)$, such that $(T(V),\shuffle,\blacklozenge,\Delta)=T(V,f)$.
\end{prop}

\begin{proof} Let $f=f_{T(V)}$. We denote by $\bullet$ the preLie product of $T(V,f)$. Let us prove that for any $x=x_1\ldots x_k, y=y_1\ldots y_l\in T(V)$, 
$x\bullet y=x\blacklozenge y$. If $k=0$, we obtain $1 \bullet y=1\blacklozenge y=0$. 
We now treat the case $l=0$. We proceed by induction on $k$. It is already done for $k=0$. If $k=1$, then $x\in V$
and $x\bullet 1=f(x)=x\blacklozenge1$. Let us assume the result at all ranks $<k$, with $k\geq 2$. Then, as the length of $x'$ and$x''$ is $<k$:
\begin{align*}
\Delta(x\bullet 1)&=x^{(1)}\otimes x^{(2)}\bullet 1+x^{(1)}\bullet 1\otimes x^{(2)}\\
&=1 \otimes x\bullet1+x'\bullet 1 \otimes 1+x'\otimes x''\bullet 1+x\otimes 1 \otimes 1\\
&=1 \otimes x\bullet1+x'\blacklozenge 1 \otimes 1+x'\otimes x''\blacklozenge 1+x\otimes 1 \otimes 1\\
&=\Delta(x\blacklozenge 1)+(x\bullet y-x\blacklozenge y)\otimes 1+1 \otimes (x\bullet y-x\blacklozenge y).
\end{align*}
We deduce that $x\bullet 1-x \blacklozenge 1$ is primitive, so belongs to $V$. As it is homogeneous of length $k\geq 2$, it is zero,
and $x\bullet 1=x\blacklozenge1$. 

We can now assume that $k,l\geq 1$. We proceed by induction on $k+l$. There is nothing left to do for $k+l=0$ or $1$.
Let us assume that the result is true at all rank $<k+l$, with $k+l\geq 2$. 
Then, using the induction hypothesis, as $x'$ and $x''$ have lengths $<k$ and $y'$ has a length $<l$:
\begin{align*}
\Delta(x\bullet y)&=1\otimes x\bullet y+x'\otimes x''\bullet y+x\otimes 1\bullet y
+x\bullet 1\otimes y+x'\bullet 1\otimes x''\shuffle y+1\bullet 1\otimes x\shuffle y\\
&+x\bullet y\otimes 1+x'\bullet y\otimes x''+1\bullet y\otimes x
+x\bullet y'\otimes y''+x'\bullet y'\otimes x''\shuffle y''+1\bullet y'\otimes x\shuffle y''\\
&=1\otimes x\bullet y+x'\otimes x''\blacklozenge y+x\otimes 1\blacklozenge y
+x\blacklozenge 1\otimes y+x'\blacklozenge 1\otimes x''\shuffle y+1\blacklozenge 1\otimes x\shuffle y\\
&+x\bullet y\otimes 1+x'\blacklozenge y\otimes x''+1\blacklozenge y\otimes x
+x\blacklozenge y'\otimes y''+x'\blacklozenge y'\otimes x''\shuffle y''+1\blacklozenge y'\otimes x\shuffle y''\\
&=\Delta(x\blacklozenge y)+(x\bullet y-x\blacklozenge y)\otimes +1\otimes (x\bullet y-x\blacklozenge y).
\end{align*}
We deduce that $x\bullet y-x\blacklozenge y$ is primitive, hence belongs to $V$. As it belongs to $V^{\otimes (k+l)}$ and $k+l\geq 2$,
it is zero. Finally, $x\bullet y=x\blacklozenge y$. \end{proof}

\begin{prop}
The Com-PreLie bialgebras $T(V,f)$ and $T(W,g)$ are isomorphic if, and only if, there exists a linear isomorphism $F:V\longrightarrow W$,
such that $g\circ F=F\circ f$. 
\end{prop}

\begin{proof}
If such an $F$ exists, by functoriality $T(F)$ is an isomorphism from $T(V,f)$ to $T(W,g)$. Let us assume that $\phi:T(V,f)\longrightarrow T(V,g)$
is an isomorphism of Com-PreLie bialgebras. Then $\phi(1)=1$, and $\phi$ induces an isomorphism from $V=Prim(T(V))$ to $W=Prim(T(W))$, 
denoted by $F$. For all $x\in V$:
$$\phi(x\bullet 1)=\phi(f(x))=F\circ f(x)=F(x)\bullet 1=g\circ F(x).$$
So such an $F$ exists. \end{proof}

\subsection{Com-PreLie algebra attached to a linear form}

Let $V$ be a a vector space, $f:V\longrightarrow \K$ be a linear form, and $\lambda \in \K$.

\begin{theo}\label{7}
Let $\bullet$ be the product on $T(V)$ such that for all $x_1,\ldots,x_m, y_1,\ldots, y_n \in V$:
$$x_1\ldots x_m \bullet y_1\ldots y_n=\sum_{i=0}^{n-1} \lambda^i f(x_1)f(y_1)\ldots f(y_i) x_2\ldots x_m \shuffle y_{i+1}\ldots y_n.$$
Then $(T(V),\shuffle,\bullet)$ is a Com-PreLie algebra. It is denoted by $T(V,f,\lambda)$.
\end{theo}

{\bf Examples.} If $x_1,x_2,x_3\in V$, $w\in T(V)$:
\begin{align*}
x_1 \bullet w&=f(x_1) w,\\
x_1x_2\bullet w&=x_1 f(x_2)w+f(x_1)(x_2 \shuffle w),\\
x_1x_2x_3 \bullet w&=x_1x_2f(x_3)w+x_1f(x_2)(x_3 \shuffle w)+f(x_1)(x_2x_3\shuffle w).
\end{align*}

In particular if $x_1=\ldots=x_n=y_1=\ldots=y_n=x$:

\begin{lemma}\label{8}
Let $x \in V$. We put $f(x)=\nu$ and $\mu=\lambda f(x)$. Then, for all $m,n\geq 0$, in $T(V,f,\lambda)$:
$$x^m \bullet x^n=\nu \sum_{j=m}^{m+n-1} \mu^{m+n-j-1} \binom{j}{m-1}x^j.$$
\end{lemma}

The proof of theorem \ref{7} will use definition \ref{9} and lemma \ref{10}:

\begin{defi}\label{9}
Let $\partial$ and $\phi$ be the linear maps defined by:
\begin{align*}
\partial:&\left\{\begin{array}{rcl}
T(V)&\longrightarrow& T(V)\\
1&\longrightarrow&0,\\
x_1\ldots x_n&\longrightarrow&f(x_1)x_2\ldots x_n,
\end{array}\right.&
\phi:&\left\{\begin{array}{rcl}
T(V)&\longrightarrow& T(V)\\
1&\longrightarrow&0,\\
x_1\ldots x_n&\longrightarrow&\displaystyle\sum_{i=0}^{n-1}\lambda^i f(x_1)\ldots f(x_i)x_{i+1}\ldots x_n.
\end{array}\right. \end{align*}\end{defi}

\begin{lemma} \label{10}\begin{enumerate}
\item For all $u,v \in T(V)$:
\begin{enumerate}
\item $\partial(u \shuffle v)=\partial(u)\shuffle v+u\shuffle \partial(v)$.
\item $\partial \circ \phi(u) \shuffle \phi(v)-\phi(\partial(u) \shuffle \phi(v))=\partial \circ \phi(v) \shuffle \phi(u)-\phi(\partial(v) \shuffle \phi(u))$.
\end{enumerate}
\item For all $u\in T(V,f,\lambda)$:
\begin{align*}
\Delta \circ \partial(u)&=(\partial \otimes Id)\circ \Delta(u),&
\Delta \circ \phi(u)&=(\phi\otimes Id)\circ \Delta(u)+1 \otimes \phi(u).
\end{align*}\end{enumerate}\end{lemma}

\begin{proof} $1$. $(a)$ This is obvious if $u=1$ or $v=1$, as $\partial(1)=0$. Let us assume that $u,v$ are nonempty words.
We put $v=xu'$,$v=yv'$, with $x,y \in V$. Then:
\begin{align*}
\partial(u\shuffle v)&=\partial(x(u'\shuffle v)+y(u\shuffle v'))\\
&=f(x)u'\shuffle v+f(y)u\shuffle v'\\
&=(f(x)u') \shuffle v+u \shuffle (f(y)v')\\
&=\partial(u)\shuffle v+u\shuffle \partial(v).
\end{align*}

$1$. $(b)$ Let us take $u=x_1\ldots x_m$ and $y=y_1\ldots y_n$ be two words of $T(V)$ of respective lengths $m$ and $n$. 
First, observe that $\phi(\partial u \shuffle \phi(v))$ is a linear span of terms:
$$\lambda^{i+j-1}f(x_1)\ldots f(x_i)f(y_1)\ldots f(y_j) x_{i+1}\ldots x_m\shuffle y_{j+1}\ldots y_m,$$
with $1\leq i\leq m$, $0\leq j\leq n$, $(i,j) \neq (0,0)$. Let us compute the coefficient of such a term:
\begin{itemize}
\item If $j<n$, it is $\displaystyle \sum_{p=0}^j\binom{i-1+j-p}{i-1}=\sum_{p=i-1}^{i+j-1} \binom{p}{i-1}=\binom{i+j}{i}$.
\item If $j=n$, its is $\displaystyle \sum_{p=0}^{n-1}\binom{i-1+j-p}{i-1}=\sum_{p=i}^{i+j-1} \binom{p}{i-1}=\sum_{p=i-1}^{i+j-1} \binom{p}{i-1}-1=\binom{i+j}{i}-1$.
\end{itemize}
We obtain:
\begin{align*}
\phi(\partial u \shuffle \phi(v))&=\sum_{i=1}^m \sum_{j=0}^n \lambda^{i+j-1} \binom{i+j}{i}f(x_1)\ldots f(x_i)f(y_1)\ldots f(y_j)x_{i+1}\ldots x_m\shuffle
y_{j+1}\ldots y_n\\
&-\sum_{i=1}^{m-1}\lambda^{i+n-1}f(x_1)\ldots f(x_i)f(y_1)\ldots f(y_n)x_{i+1}\ldots x_m\\
&-\lambda^{m+n-1}\binom{m+n}{m}f(x_1)\ldots f(x_m)f(y_1)\ldots f(y_n)\\
&=\sum_{i=1}^m \sum_{j=1}^n \lambda^{i+j-1} \binom{i+j}{i}f(x_1)\ldots f(x_i)f(y_1)\ldots f(y_j)x_{i+1}\ldots x_m\shuffle
y_{j+1}\ldots y_n\\
&+\sum_{i=1}^m\lambda^{i-1}f(x_1)\ldots f(x_i)x_{i+1}\ldots x_m\shuffle y_1\ldots y_n\\
&-\sum_{i=1}^{m-1}\lambda^{i+n-1}f(x_1)\ldots f(x_i)f(y_1)\ldots f(y_n)x_{i+1}\ldots x_m\\
&-\lambda^{m+n-1}\binom{m+n}{m}f(x_1)\ldots f(x_m)f(y_1)\ldots f(y_n).
\end{align*}
Moreover:
\begin{align*}
\partial \circ \phi(u)\shuffle \phi(v)&=\sum_{i=1}^m\sum_{j=0}^{n-1}\lambda^{i+j-1} \binom{i+j}{i}f(x_1)\ldots f(x_i)f(y_1)\ldots f(y_j)x_{i+1}\ldots x_m\shuffle
y_{j+1}\ldots y_n\\
&=\sum_{i=1}^{m-1}\sum_{j=1}^{n-1}\lambda^{i+j-1} \binom{i+j}{i}f(x_1)\ldots f(x_i)f(y_1)\ldots f(y_j)x_{i+1}\ldots x_m\shuffle
y_{j+1}\ldots y_n\\
&+\sum_{j=1}^{n-1} \lambda^{j+m-1}f(x_1)\ldots f(x_m)f(y_1)\ldots f(y_j)y_{j+1}\ldots y_n\\
&+\sum_{i=1}^m\lambda^{i-1}f(x_1)\ldots f(x_i)x_{i+1}\ldots x_m\shuffle y_1\ldots y_n.
\end{align*}
Hence:
\begin{align*}
&\partial \circ \phi(u)\shuffle \phi(v)-\phi(\partial u \shuffle \phi(v))\\
&=\sum_{i=1}^{m-1}\sum_{j=1}^{n-1}\lambda^{i+j-1} \binom{i+j}{i}f(x_1)\ldots f(x_i)f(y_1)\ldots f(y_j)x_{i+1}\ldots x_m\shuffle
y_{j+1}\ldots y_n\\
&-\sum_{i=1}^m \sum_{j=1}^n \lambda^{i+j-1} \binom{i+j}{i}f(x_1)\ldots f(x_i)f(y_1)\ldots f(y_j)x_{i+1}\ldots x_m\shuffle
y_{j+1}\ldots y_n\\
&+\lambda^{m+n-1}\binom{m+n}{m}f(x_1)\ldots f(x_m)f(y_1)\ldots f(y_n)\\
&+\sum_{j=1}^{n-1} \lambda^{j+m-1}f(x_1)\ldots f(x_m)f(y_1)\ldots f(y_j)y_{j+1}\ldots y_n\\
&+\sum_{i=1}^{m-1}\lambda^{i+n-1}f(x_1)\ldots f(x_i)f(y_1)\ldots f(y_n)x_{i+1}\ldots x_m.
\end{align*}
The three first rows are symmetric in $u$ and $v$, whereas the sum of the fourth and fifth rows is symmetric in $u$ and $v$.
So $\partial \circ \phi(u)\shuffle \phi(v)-\phi(\partial u \shuffle \phi(v))$ is symmetric in $u$ and $v$.\\

2. Let us take $u=x_1\ldots x_n$, with $x_1,\ldots,x_n \in V$. Then:
\begin{align*}
\Delta \circ \partial(u)&=f(x_1)\sum_{i=1}^n x_2\ldots x_i \otimes x_{i+1}\ldots x_n\\
&=\sum_{i=1}^n \partial(x_1\ldots x_i)\otimes x_{i+1}\ldots x_n+\partial(1)\otimes x_1\ldots x_n\\
&=\sum_{i=0}^n \partial(x_1\ldots x_i)\otimes x_{i+1}\ldots x_n\\
&=(\partial \otimes Id)\circ \Delta(u).
\end{align*}
Moreover:
\begin{align*}
\Delta \circ \phi(u)&=\sum_{i=0}^{n-1}\lambda^if(x_1)\ldots f(x_i)\Delta(x_{i+1}\ldots x_n)\\
&=\sum_{i=0}^{n-1}\sum_{j=i}^n \lambda^i f(x_1)\ldots f(x_i)x_{i+1}\ldots x_j\otimes x_{j+1}\ldots x_n\\
&=\sum_{j=0}^n \sum_{i=0}^j  \lambda^i f(x_1)\ldots f(x_i)x_{i+1}\ldots x_j\otimes x_{j+1}\ldots x_n-\lambda^n f(x_1)\ldots f(x_n)\otimes 1\\
&=\sum_{j=0}^n \phi(x_1\ldots x_j) \otimes x_{j+1}\ldots x_n+\sum_{j=0}^{n-1} \lambda^j f(x_1)\ldots f(x_j)\otimes x_{j+1}\ldots x_n\\
&=\sum_{j=0}^n \phi(x_1\ldots x_j) \otimes x_{j+1}\ldots x_n+1 \otimes \left(\sum_{j=0}^{n-1} \lambda^j f(x_1)\ldots f(x_j)x_{j+1}\ldots x_n\right)\\
&=(\phi \otimes Id)\circ \Delta(u)+1 \otimes \phi(u).
\end{align*}\end{proof}

\begin{proof} (Theorem \ref{7}). By definition, for all $u,v \in T(V)$:
$$u\bullet v=\partial(u) \shuffle \phi(v).$$
Let $u,v,w \in T(V)$. By lemma \ref{10}-1:
\begin{align*}
(u\shuffle v) \bullet w&=\partial(u\shuffle v) \shuffle \phi(w)\\
&=\partial(u)\shuffle v \shuffle \phi(w)+u\shuffle \partial(v) \shuffle \phi(w)\\
&=\partial(u) \shuffle \phi(w)\shuffle v+u\shuffle \partial(v) \shuffle \phi(w)\\
&=(u\bullet w)\shuffle v+u\shuffle (v\bullet w).
\end{align*}
Moreover:
\begin{align*}
(u\bullet v)\bullet w-u\bullet (v\bullet w)&=(\partial(u)\shuffle \phi(v))\bullet w-u\bullet(\partial(v)\shuffle \phi(w))\\
&=\partial(\partial (u)\shuffle \phi(v))\shuffle \phi(w)-\partial(u)\shuffle \phi(\partial(v)\shuffle \phi(w))\\
&=\partial^2(u)\shuffle \phi(v)\shuffle \phi(w)+\partial(u)\shuffle \left(\partial \circ \phi(v)\shuffle \phi(w)- \phi(\partial(v)\shuffle \phi(w))\right).
\end{align*}
By lemma \ref{10}-2, this is symmetric in $v$ and $w$. Consequently, $T(V,f,\lambda)$ is Com-PreLie. \end{proof}\\

This construction is functorial. Let $(V,f)$ and $(W,g)$ be two spaces equipped with a linear form
and let $F:V\longrightarrow W$ be a map such that $g\circ F=f$. Then $T(F)$ is a Com-PreLie algebra morphism from $T(V,f,\lambda)$ to $T(W,g,\lambda)$. 

\begin{prop}
$(coS(V),\shuffle,\bullet,\Delta)$ is a Com-PreLie bialgebra, denoted by $coS(V,f,\lambda)$.
\end{prop}

\begin{proof} Let us first prove that $coS(V)$ is stable under $\bullet$. It is enough to prove that it is stable under $\partial$ and $\phi$.
Let us first consider $\partial$. As it is a derivation for $\shuffle$, it is enough to prove that $\partial(V) \subseteq coS(V)$, which is obvious
as $\partial(V)\subset \K$. Let us now consider $\phi$. Let $x_1,\ldots,x_k \in V$.
\begin{align*}
\phi(x_1\shuffle \ldots \shuffle x_k)&=\sum_{\sigma \in \mathfrak{S}_k} \phi(x_{\sigma(1)}\ldots x_{\sigma(k)})\\
&=\sum_{i=0}^{k-1} \sum_{\sigma \in \mathfrak{S}_k} \mu^if(x_{\sigma(1)})\ldots f(x_{\sigma(i)})x_{\sigma(i+1)}\ldots x_{\sigma(k)}\\
&=\sum_{i=0}^{k-1} \sum_{1\leq k_1<\ldots<k_i\leq k} i!\mu^i\prod_{j=1}^i f(x_{k_i}) x_1\shuffle \widehat{x_{k_1}}\shuffle\ldots \shuffle 
\widehat{x_{k_i}}\shuffle\ldots \shuffle x_k. 
\end{align*}
This is an element of $coS(V)$, so $coS(V)$ is stable under $\bullet$.\\

Let us prove now the compatibility between $\bullet$ and the coproduct of $coS(V)$.
As $coS(V)$ is cocommutative, lemma \ref{10} implies that for all $u\in coS(V)$:
$$\Delta \circ \partial(u)=\partial(u^{(1)})\otimes u^{(2)}=\partial(u^{(2)})\otimes u^{(1)}=u^{(1)}\otimes \partial(u^{(2)}).$$
Let us consider $u,v \in coS(V)$. Then, by lemma \ref{10}:
\begin{align*}
\Delta(u\bullet v)&=\Delta(\partial (u) \shuffle \phi(v))\\
&=(\Delta\circ \partial(u))\shuffle \Delta \circ \phi(v)\\
&=(\Delta\circ \partial u)\shuffle(\Phi(v^{(1)})\otimes v^{(2)}+1 \otimes \phi(v))\\
&=\partial(u^{(1)})\shuffle \Phi(v^{(1)})\otimes u^{(2)}\shuffle v^{(2)}+u^{(1)}\shuffle 1 \otimes \partial(u^{(2)})\shuffle \phi(v)\\
&=u^{(1)}\bullet v^{(1)}\otimes u^{(2)}\shuffle v^{(2)}+u^{(1)}\otimes u^{(2)}\bullet v.
\end{align*}
So $coS(V)$ is a Com-PreLie bialgebra. \end{proof}\\

Note that $f_{coS(V,f,\lambda)}=0$. The preLie product induced on $Prim(coS(V))=V$ is given by $x\star y=f(x)y$.

\begin{cor}
Let $V$ be a vector space, $f\in V^*$, $\lambda \in \K$. We give $S(V)$ its usual product $m$ and coproduct $\bfDelta$,
defined by $\bfDelta(v)=v\otimes 1+1\otimes v$ for all $v\in V$, and the product $\bullet$ defined by:
\begin{enumerate}
\item $1\bullet x=0$ for any $x \in S(V)$.
\item $\displaystyle x\bullet x_1\ldots x_k=\sum_{I\subsetneq \{1,\ldots,k\}} |I|! \lambda^{|I|} f(x) \prod_{i\in I} f(x_i) \prod_{i\notin I} x_i$, 
for all $x,x_1,\ldots,x_k \in V$.
\item $\displaystyle x_1\ldots x_k\bullet x=\sum_{i=1}^k x_1\ldots (x_i \bullet x)\ldots x_k$ for any $x_1,\ldots,x_k \in V$, $x\in S(V)$. 
\end{enumerate} 
Then $(S(V),m,\bullet,\bfDelta)$ is a Com-PreLie bialgebra, denoted by $S(V,f,\lambda)$. 
\end{cor}

\begin{proof} There is a Hopf algebra isomorphism:
$$\theta:\left\{\begin{array}{rcl}
(S(V),m,\bfDelta)&\longrightarrow&(coS(V),\shuffle,\Delta)\\
v\in V&\longrightarrow&v.
\end{array}\right.$$
Let $v,x_1,\ldots,x_k \in V$. 
\begin{align*}
\theta(v)\bullet \theta(x_1\ldots x_k)&=v\bullet x_1\shuffle \ldots \shuffle x_k\\
&=f(v)\shuffle \phi(x_1\shuffle \ldots \shuffle x_k)\\
&=f(v)\sum_{i=0}^{k-1} \sum_{1\leq k_1<\ldots<k_i\leq k} i!\mu^i\prod_{j=1}^i f(x_{k_i}) 
x_1\shuffle \widehat{x_{k_1}}\shuffle\ldots \shuffle \widehat{x_{k_i}}\shuffle\ldots \shuffle x_k\\
&=\theta\left(\sum_{I\subsetneq \{1,\ldots,k\}}|I|!\mu^i\prod_{i\in I} f(x_i) \prod_{i\notin I} x_i\right).
\end{align*}
Therefore, as $coS(V)$ is a Com-PreLie algebra, $S(V)$ is also a Com-PreLie bialgebra. \end{proof}

\begin{prop}\label{13}
Let us assume that $f \neq 0$. Then:
\begin{enumerate}
\item $(T(V),\prec,\bullet)$ is a Zinbiel-PreLie algebra if, and only if, $dim(V)=1$.
\item $(T(V),\shuffle,\bullet,\Delta)$ is a Com-PreLie bialgebra if, and only if, $dim(V)=1$.
\end{enumerate}\end{prop}

\begin{proof}
1. $\Longrightarrow$. Let $y\in V$, such that $f(y)=1$. Note that $y\neq 0$. Let $x \in V$, such that $f(x)=0$. Then:
\begin{align*}
(x\prec y)\bullet y&=xy\bullet y=f(x) y\shuffle y=0,\\
(x\bullet y)\prec y+x\prec (y\bullet y)&=f(x)y\prec y+x \prec f(y)y=0+f(y) x\prec y=xy. 
\end{align*}
As $T(V,f,\lambda)$ is Zinbiel-PreLie, $xy=0$. As $y\neq 0$, $x=0$; we obtain that $f$ is injective, so $dim(V)=1$.\\

1. $\Longrightarrow$. We use the notations of lemma \ref{8}. It is enough to prove that for all $k,l,m \geq 0$,
$(x^k \prec x^l)\bullet x^m=(x^k \bullet x^m) \prec x^l+x^k \prec (x^l \bullet x^m)$.
$$(x^k\prec x^l)\bullet x^m=\lambda\sum_{j=k+l}^{k+l+m-1} \mu^{k+l+m-j-1} \binom{j}{k+l-1}\binom{k+l-1}{k-1}x^j,$$
and:
\begin{align*}
&(x^k\bullet x^m)\prec x^l+x^k \prec (x^l\bullet x^m)\\
&=\lambda\sum_{j=k}^{k+m-1}\mu^{k+m-j-1}\binom{j}{k-1}x^j \prec x^l+\lambda\sum_{j=l}^{l+m-1}\mu^{l+m-1-j}\binom{j}{k-1}x^k\prec x^j\\
&=\lambda\sum_{j=k}^{k+m-1}\mu^{k+m-j-1}\binom{j}{k-1}\binom{j+l-1}{j-1}x^{l+j}
+\lambda \sum_{j=l}^{l+m-1}\mu^{l+m-1-j}\binom{j}{k-1}\binom{k+j-1}{k-1}x^{k+j}\\
&=\lambda\sum_{j=k+l}^{k+l+m-1}\mu^{k+l+m-j-1}\binom{j-l}{k-1}\binom{j-1}{j-l-1}x^j
+\lambda\sum_{j=k+l}^{k+l+m-1}\mu^{k+l+m-j-1}\binom{j-k}{l-1}\binom{j-1}{k-1}x^j.
\end{align*}
Moreover, a simple computation proves that:
$$\binom{j-l}{k-1}\binom{j-1}{j-l-1}+\binom{j-k}{l-1}\binom{j-1}{k-1}=\binom{j}{k+l-1}\binom{k+l-1}{k-1}.$$
So $T(V,f,\lambda)$ is Zinbiel-PreLie.\\

2. $\Longrightarrow$. Let us choose $z\in V$, nonzero, and $x \in V$ such that $f(x)=1$. Then:
$$\Delta(xy\bullet z)=\Delta(f(x) y\shuffle z)=xy\bullet z\otimes 1+1 \otimes xy\bullet z+y\otimes z+z\otimes y,$$
whereas:
\begin{align*}
&(xy)^{(1)}\otimes (xy)^{(2)}\bullet z+(xy)^{(1)}\bullet z^{(1)}\otimes (xy)^{(2)}\shuffle z^{(2)}\\
&=xy\otimes 1 \bullet z+x\otimes y\bullet z+1 \otimes xy \bullet z\\
&+xy\bullet z\otimes 1+xy\bullet 1\otimes z+x\bullet z\otimes y+x\bullet 1 \otimes y\shuffle z+1 \bullet z\otimes xy
+1\bullet 1 \otimes xy\shuffle z\\
&=xy\bullet z\otimes 1+1 \otimes xy\bullet z+f(y)x\otimes z+z\otimes y.
\end{align*}
So, for all $y\in V$, $f(y)x\otimes z=y\otimes z$. As $z\neq 0$, $f(y)x=y$: $V=Vect(x)$ is one-dimensional.\\

$\Longrightarrow$. In this case, $T(V)=coS(V)$, so is a Com-PreLie bialgebra. \end{proof}

\begin{prop}
The Com-PreLie bialgebras $coS(V,f,\lambda)$ and $coS(W,g,\mu)$ are isomorphic if, and only if, one of the following assertion holds:
\begin{enumerate}
\item $dim(V)=dim(W)$, and $f$ and $g$ are both zero.
\item $dim(V)=dim(W)$, $\lambda=\mu$ and $f$ and $g$ are  both nonzero.
\end{enumerate}
\end{prop}

\begin{proof} If $dim(V)=dim(W)$, and $f$ and $g$ are both zero, then $\bullet=0$ in both these Com-preLie bialgebras.
Take any linear isomorphism $F$ from $V$ to $W$, then the restriction of $T(F)$ as an algebra morphism from $coS(V)$ to $coS(W)$
is an isomorphism of Com-PreLie bialgebras.

If $dim(V)=dim(W)$, $\lambda=\mu$ and $f$ and $g$ are  both nonzero, there exists an isomorphism $F:V\longrightarrow W$
such that $g\circ F=f$. By functoriality, $T(V,f,\lambda)$ and $T(W,g,\lambda)$ are isomorphic via $T(F)$. The restriction of $T(F)$
induces an isomorphism from $coS(V,f,\lambda)$ to $coS(W,g,\lambda)$.\\

Let us assume that $\phi:coS(V,f,\lambda)\longrightarrow coS(W,g,\mu)$ is an isomorphism of Com-PreLie bialgebras.
It induces an isomorphism from $Prim(coS(V))=V$ to $Prim(coS(W))=W$, denoted by $F$: consequently, $dim(V)=dim(W)$.
Let us choose $y\in V$, nonzero. For all $x\in V$:
$$\phi(x\bullet y)=\phi(f(x)y)=f(x)F(y)=\phi(x)\bullet \phi(y)=F(x)\bullet F(y)=g\circ F(x)F(y).$$
As $F$ is an isomorphism, for all $x\in V$, $f(x)=g\circ F(x)$. So $f$ and $g$ are both zero or are both nonzero.
Let us assume that they are nonzero. We choose $x\in V$, such that $f(x)=1$. Then:
$$\phi(x^2)=\phi\left(\frac{x\shuffle x}{2}\right)=\frac{\phi(x)\shuffle \phi(x)}{2}=F(x)^2.$$
Hence:
\begin{align*}
\phi(x)\bullet \phi(x^2)&=F(x)\bullet F(x)^2&\phi(x\bullet x^2)&=\phi(f(x)x^2+\lambda f(x)^2x)\\
&=g\circ F(x) F(x)^2+\mu g\circ F(x)^2 F(x)&&=F(x)^2+\lambda F(x).\\
&=F(x)^2+\mu F(x).
\end{align*}
As $x\neq 0$, $F(x)\neq 0$, so $\lambda=\mu$. \end{proof}

\subsection{Com-PreLie algebra associated to a preLie algebra}

\begin{theo}
Let $(V,\star)$ be a preLie algebra. We define a product on $T(V)$ by:
$$x_1\ldots x_k \bullet y_1\ldots y_l=\sum_{i=1}^k x_1\ldots x_{i-1} (x_i \star y_1) (x_{i+1}\ldots x_l \shuffle y_2\ldots y_l),$$
for all $x_1,\ldots,x_k,y_1,\ldots,y_l \in V$; by convention, this is equal to $0$ if $k=0$ or $l=0$.
Then $(T(V),\prec,\bullet,\Delta)$ is a Zinbiel-PreLie bialgebra, denoted by $T(V,\star)$. 
\end{theo}

{\bf Examples.} Let $x_1,x_2,x_3,y\in V$, $w\in T(V)$.
\begin{align*}
x_1\bullet yw&=(x_1\star y)w,\\
x_1x_2\bullet yw&=(x_1\star y)(x_2\shuffle w)+x_1(x_2\star y) w,\\
x_1x_2x_3\bullet yw&=(x_1\star y)(x_2x_3\shuffle w)+x_1(x_2\star y)(x_3\shuffle w)+x_1x_2(x_3\star y)w.
\end{align*}

\begin{proof} First, remark that for all $x,y\in V$, for all $u,v\in T(V)$:
$$xu\bullet yv=(x\star y)u\shuffle v+x(u\bullet yv).$$
Let us prove that for all $a,b,c\in T(V)$, $(a\prec b)\bullet c=(a\bullet c)\prec b+a\prec (b\bullet c)$.
This is obvious if one of $a,b,c$ is equal to $1$, as $1 \bullet d=d\bullet 1=0$ for all $d$.
We now assume that $a,b,c$ are nonempty words of respective lengths $k$, $l$ and $m$, and we proceed by induction on $k+l+m$.
There is nothing to do if $k+l+m\leq 2$. Let us assume the result at rank $k+l+m-1$. We put $a=xu$, $b=v$, $c=zw$, avec $x,z\in V$. 
\begin{align*}
(xu \prec v)\bullet zw&=(x(u\shuffle v))\bullet zw\\
&=x\star z(u\shuffle v\shuffle w)+x((u\shuffle v)\bullet zw)\\
&=x\star z(u\shuffle v\shuffle w)+x((u\bullet zw)\shuffle v+u\shuffle(v\bullet zw))\\
&=(x\star z(u\shuffle w)) \prec v+x(u\bullet zw)\prec v+xu\prec (v\bullet zw)\\
&=(xu\bullet v)\prec zw+xu\prec (v\bullet zw).
\end{align*}

Let us now prove that for all $a,b,c\in T(V)$, $a\bullet (b\bullet c)-(a\bullet b)\bullet c=a\bullet (c\bullet b)-(a\bullet c)\bullet b$.
If one of $a$, $b$, $c$ is equal to $1$, this is obvious. We now assume that $a,b,c$ are nonempty words of respective lengths $k$, $l$ and $m$, 
and we proceed by induction on $k+l+m$. There is nothing to do if $k+l+m\leq 2$. Let us assume the result at rank $k+l+m-1$. 
We put $a=xu$, $b=yv$, $c=zw$, avec $x,y,z\in V$. 
\begin{align*}
(xu\bullet yv)\bullet zw&=(x\star y(u\shuffle v))\bullet zw+(x(u\bullet yv))\bullet zw\\
&=(x\star y)\star z(u\shuffle v \shuffle w)+x\star y((u\shuffle v)\bullet zw)\\
&+x\star z((u\bullet yv)\shuffle w)+x((u\bullet yv)\bullet zw);\\
xu\bullet (yv\bullet zw)&=xu\bullet (y\star z(v\shuffle w)+y(v\bullet zw))\\
&=x\star (y\star z)u\shuffle v\shuffle w+x u\bullet(y\star z(v\shuffle w))\\
&+x\star y(u\shuffle (v\bullet zw))+x(u\bullet y(v\bullet zw)).
\end{align*}
Hence:
\begin{align*}
(xu\bullet yv)\bullet zw-xu\bullet (yv\bullet zw)&=((x\star y)\star z-x\star (y\star z))(u\shuffle v \shuffle w)\\
&+x\star y((u\shuffle v)\bullet zw)+x\star z((u\bullet yv)\shuffle w)\\
&+x((u\bullet yv)\bullet zw-u\bullet (yv\bullet zw)).
\end{align*}
As $\star$ is preLie and $\shuffle$ is commutative, the first row is symmetric in $yv$ and $zw$.
The second row is obviously symmetric in $yv$ and $zw$,and by the induction hypothesis, the last row also is.
So the preLie relation is satisfied for $xu$, $yv$ and $zw$.\\

Let us prove the compatibility with the coproduct. Let $a,b \in T(V)$. Let us prove that:
$$\Delta(a\bullet b)=a^{(1)}\otimes a^{(2)}\bullet b+a^{(1)}\bullet b^{(1)}\otimes a^{(2)}\shuffle b^{(2)}.$$
This is immediate if $a$ or $b$ is equal to $1$. 
We now assume that $a$ and $b$ are nonempty words of respective lengths $k$ and $l$, and we proceed by induction on $k+l$.
There is nothing to do if $k+l\leq 1$. Let us assume the result at rank $k+l-1$. We put $a=xu$ and $b=yv$, $x,y\in V$. 
\begin{align*}
\Delta(xu\bullet yv)&=\Delta(x\star y(u\shuffle v)+x(u\bullet yv))\\
&=(x\star y) u^{(1)}\shuffle v^{(1)}\otimes u^{(2)}\shuffle v^{(2)}+1\otimes x\star y(u\shuffle v)\\
&+xu^{(1)}\otimes u^{(2)}\bullet yv+x(u^{(1)}\bullet yv^{(1)})\otimes u^{(2)}\shuffle v^{(2)}\\
&+xu^{(1)}\bullet 1\otimes u^{(2)}\shuffle yv+1\otimes x(u\bullet yv)\\
&=xu^{(1)}\otimes u^{(2)}\bullet yv+1\otimes xu\bullet yv\\
&+(x\star y) u^{(1)}\shuffle v^{(1)}\otimes u^{(2)}\shuffle v^{(2)}\\
&=(xu)^{(1)}\otimes (xu)^{(2)}\bullet yv\\
&+(xu)^{(1)}\bullet (yv)^{(1)}\otimes (xu)^{(2)}\shuffle (yv)^{(2)}.
\end{align*}
So $T(V,\star)$ is indeed a Zinbiel-PreLie bialgebra. \end{proof}\\

This is also a functorial construction. If $F:(V,\star)\longrightarrow (W,\star)$ is a preLie algebra morphism, then $T(F)$
is a Zinbiel-PreLie algebra morphism.\\

Note that $f_{T(V,\star)}=0$. The preLie product induced on $Prim(T(V))=V$ is given by $\star$.

\begin{prop} 
Let $\blacklozenge$ be a product on $T(V)$, such that $(T(V),\shuffle,\bullet,\Delta)$ is a Zinbiel-PreLie bialgebra, with 
$V^{\otimes k}\blacklozenge V^{\otimes l}\subseteq V^{\otimes (k+l-1)}$ for all $k,l\in \N$. There exists a preLie product $\star$ on $V$,
such that $(T(V),\prec,\\blacklozenge,\Delta)=T(V,\star)$.
\end{prop}

\begin{proof} By hypothesis, $V\blacklozenge V\subseteq V$: $V$ is a preLie subalgebra of $T(V)$. We denote its preLie product by $\star$,
and by $\bullet$ the preLie product of $T(V,\star)$.  Let us prove that for any $x=x_1\ldots x_k, y=y_1\ldots y_l\in T(V)$, 
$x\bullet y=x\blacklozenge yl$.  If $k=0$, we obtain $1 \bullet y=1' \bullet y=0$. 
We now treat the case $l=0$: let us prove that $x\blacklozenge 1=0$ by induction on $k$. It is already done for $k=0$. If $k=1$, then $x\in V$,
so $x\blacklozenge 1\in \K$ by homogeneity. Moreover, $\varepsilon(x\blacklozenge 1)=0$, so $x\blacklozenge 1=0$.
Let us assume the result at rank $k-1$, with $k\geq 2$. We put $u=x_2\ldots x_k$. Then:
$$x\blacklozenge 1=(x_1\prec u)\blacklozenge 1=(x_1\blacklozenge 1)\prec u+x_1\prec (u\blacklozenge 1)=0+0=0.$$

We can now assume that $k,l\geq 1$. We proceed by induction on $k+l$. There is nothing to do for $k+l=0$ or $1$.
If $k+l=2$, then $k=l=1$, and $x\bullet y=x\star y=x\blacklozenge y$. Let us assume that the result is true at all rank $<k+l$, with $k+l\geq 3$. 
Then, using the induction hypothesis, as $x'$ and $x''$ have lengths $<k$ and $y'$ has a length $<l$:
\begin{align*}
\Delta(x\bullet y)&=1\otimes x\bullet y+x'\otimes x''\bullet y+x\otimes 1\bullet y+x\bullet 1\otimes y+x'\bullet 1\otimes x''\shuffle y+1\bullet 1\otimes x\shuffle y\\
&+x\bullet y\otimes 1+x'\bullet y\otimes x''+1\bullet y\otimes x+x\bullet y'\otimes y''+x'\bullet y'\otimes x''\shuffle y''+1\bullet y'\otimes x\shuffle y''\\
&=1\otimes x\bullet y+x'\otimes x''\blacklozenge y+x\otimes 1\blacklozenge y
+x\blacklozenge 1\otimes y+x'\blacklozenge 1\otimes x''\shuffle y+1\blacklozenge 1\otimes x\shuffle y\\
&+x\bullet y\otimes 1+x'\blacklozenge y\otimes x''+1\blacklozenge y\otimes x
+x\blacklozenge y'\otimes y''+x'\blacklozenge y'\otimes x''\shuffle y''+1\blacklozenge y'\otimes x\shuffle y''\\
&=\Delta(x\blacklozenge y)+(x\bullet y-x\blacklozenge y)\otimes +1\otimes (x\bullet y-x\blacklozenge y).
\end{align*}
We deduce that $x\bullet y-x\blacklozenge y$ is primitive, so belongs to $V$. As it belongs to $V^{\otimes (k+l-1)}$ and $k+l-1\geq 2$,
it is zero. So $x\bullet y=x\blacklozenge y$. \end{proof}

\begin{prop}\begin{enumerate}
\item Let $(V,\star)$ and $(V',\star')$ be two preLie algebras. The Com-PreLie bialgebras $T(V,\star)$ and $T(V',\star')$ are isomorphic if,
and only if, the preLie algebras  $(V,\star)$ and $(V',\star')$ are isomorphic.
\item Let $(V,\star)$ be a preLie algebra and $g:W\longrightarrow W$ be an endomorphism. The Com-PreLie bialgebras $T(V,\star)$ and $T(W,g)$ 
are isomorphic if, and only if, $dim(V)=dim(W)$, $\star=0$ and $f=0$.
\end{enumerate}\end{prop}

\begin{proof} 1. If $F:V\longrightarrow V'$ is a preLie algebra isomorphism, by functoriality, $T(F)$ is an isomorphism from
$T(V,\star)$ to $T(V',\star')$. Let us assume that $\phi:T(V,\star)\longrightarrow T(V',\star')$ is an isomorphism.
It induces by restriction an isomorphism $F$ from $Prim(T(V))=V$ to $Prim(T(V'))=V'$. Moreover, for all $x,y\in V$:
$$\phi(x\bullet y)=\phi(x\star y)=F(x\star y)=\phi(x)\bullet \phi(y)=F(x)\bullet F(y)=F(x)\star'F(y).$$
So $(V,\star)$ and $(V',\star')$ are isomorphic. \\

2. If $dim(V)=dim(W)$, $\star=0$ and $f=0$, then both preLie product of $T(V,\star)$ and $T(W,g)$ are zero.
Let $F:V\longrightarrow W$ be an isomorphism. Then $T(F)$ is an isomorphism from $T(V,\star)$ to $T(W,g)$.
Conversely, if $\phi:T(V,\star)\rightarrow T(W,g)$ is an isomorphism, it induces an isomorphism $F$ from $Prim(T(V))=V$ to $Prim(T(W))=W$.
As $\phi(1)=1$, for all $x\in V$:
$$\phi(x\bullet 1)=0=\phi(x)\bullet \phi(1)=F(x)\bullet 1=g\circ F(x).$$
As $F$ is an isomorphism, $g=0$, so the preLie product of $T(W,g)$ is zero. By isomorphism, the preLie product $\star$ of $T(V,\star)$ is zero. \end{proof}

\section{Examples on $\K[X]$}

Our aim in this section is to give all preLie products on $\K[X]$, making it a graded Com-PreLie algebra.
We shall prove the following result:

\begin{theo}\label{18}
\begin{enumerate}
\item The following objects are Zinbiel-PreLie algebras:
\begin{enumerate}
\item Let $N\geq 1$, $\lambda,a,b \in \K$, $a\neq 0$, $b\notin \mathbb{Z}_-$. We put $\g^{(1)}(N,\lambda,a,b)=(\K[X],m,\bullet)$, with:
$$X^i \bullet X^j=\begin{cases}
i\lambda X^i \mbox{ if }j=0,\\
a\frac{i}{\frac{j}{N}+b} X^{i+j}\mbox{ if }j\neq 0 \mbox{ and }N\mid j,\\
0\mbox{ otherwise}.
\end{cases}$$ 
\item Let $N\geq 1$, $\lambda,\mu \in \K$, $\mu \neq 0$. We put $\g^{(2)}(N,\lambda,\mu)=(\K[X],m,\bullet)$, with:
$$X^i \bullet X^j=\begin{cases}
i\lambda X^i \mbox{ if }j=0,\\
i\mu X^{i+N} \mbox{ if }j=N,\\
0\mbox{ otherwise}.
\end{cases}$$
\item Let $N\geq 1$, $\lambda,\mu \in \K$, $\mu \neq 0$. We put $\g^{(3)}(N,\lambda,\mu)=(\K[X],m,\bullet)$, with:
$$X^i \bullet X^j=\begin{cases}
i\lambda X^i \mbox{ if }j=0,\\
i\mu X^{i+j}  \mbox{ if }j\neq 0 \mbox{ and }N\mid j,\\
0\mbox{ otherwise}.
\end{cases}$$
\item Let $\lambda \in \K$. We put $\g^{(4)}(\lambda)=(\K[X],m,\bullet)$, with:
$$X^i \bullet X^j=\begin{cases}
i\lambda X^i \mbox{ if }j=0,\\
0\mbox{ otherwise}.
\end{cases}$$ 
In particular, the preLie product of $\g^{(4)}(0)$ is zero.
\end{enumerate}
\item Moreover, if  $\bullet$ is a product on $\K[X]$, such that $\g=(\K[X],m,\bullet)$ is a graded Com-PreLie algebra, Then $\g$ is one of the 
preceding examples. 
\end{enumerate}\end{theo}

{\bf Remark.} If $\lambda=\frac{a}{b}$, in $\g^{(1)}(N,\lambda,a,b)$, for all $i,j \in \N$:
$$X^i \bullet X^j=\begin{cases}
\frac{ai}{\frac{j}{N}+b} \mbox{ if }N\mid j,\\
0\mbox{ otherwise}.
\end{cases}$$
We put $\g^{(1)}(N,a,b)=\g^{(1)}(N,\frac{a}{b},a,b)$.\\

It is possible to prove that all these Com-PreLie algebras are not isomorphic. However, they can be isomorphic as Lie algebras.
Let us first recall some notations on the Faà di Bruno Hopf algebra \cite{FoissyDSE2}:
\begin{itemize}
\item $\g_{FdB}$ has a basis $(e_i)_{i\geq 1}$, and for all $i,j \geq 1$, $[e_i,e_j]=(i-j)e_{i+j}$.
\item Let $\alpha \in \K$. The right $\g_{FdB}$-module has a basis $V_\alpha=Vect(f_i)_{i\geq 1}$, and the right action of $\g_{FdB}$ is defined by
$f_i.e_j=(i+\alpha)e_{i+j}$.
\end{itemize}

\begin{prop}
Let $N\geq 1$, $\lambda,\lambda',\mu,a,b \in \K$, $\mu,a\neq 0$, $b\notin \mathbb{Z}_-$. Then, as Lie algebras:
$$\g^{(1)}(N,\lambda,a,b)_+\approx \g^{(3)}(N,\lambda,\mu)_+\approx \left(V_{-\frac{1}{N}}\oplus\ldots \oplus V_{-\frac{N-1}{N}}\right) \rtimes \g_{FdB}.$$
\end{prop}

\begin{proof} We first work in $\g^{(1)}(N,\lambda,a,b)$. For all $i\geq 1$, for all $1\leq r\leq N-1$, we put
$E_i=\frac{i+b}{Na} X^{Ni}$ and $F_i^{(r)}=X^{N(i-1)+r}$. Then $\displaystyle (E_i)_{i\geq 1}\cup \bigcup_{r=1}^{N-1} \left(F_i^{(r)}\right)_{i\geq 1}$
is a basis of $\g^{(1)}(N,\lambda,a,b)_+$, and, for all $i,j \geq 1$, for all $1\leq r,s\leq N-1$:
\begin{align*}
[E_i,E_j]&=(i-j)E_{i+j},&[F_i^{(r)},F_j^{(s)}]&=0,& [F_i^{(r)},E_j]&=\left(i+\frac{r-N}{N}\right)F_{i+j}^{(r)}.
\end{align*}
Hence, this Lie algebra is isomorphic to $\left(V_{-\frac{N-1}{N}}\oplus\ldots \oplus V_{-\frac{1}{N}}\right) \rtimes \g_{FdB}$.
The proof is similar for $\g^{(3)}(N,\lambda',\mu)$, with $E_i=\frac{1}{N\mu} X^{Ni}$ and $F_i^{(r)}=X^{N(i-1)+r}$. \end{proof}\\

Consequently, we can describe the group corresponding to these Lie algebras.
\begin{enumerate}
\item $G_{FdB}$ is the group of formal diffeomorphisms of $\K$ tangent to the identity:
$$G_{FdB}=(\{X+a_1X^2+a_2X^3+\ldots\mid a_2,a_2,\ldots \in \K\},\circ).$$
\item For all $\alpha \in \K$, we define a right $G_{FdB}$-module $\mathbb{V}_\alpha$: as a vector space, this is $\K[[X]]_+$.
The action is given by $P.Q=\left(\frac{Q(X)}{X}\right)^\alpha P\circ Q(X)$ for all $P\in \mathbb{V}_\alpha$ and $Q \in G_{FdB}$.
\end{enumerate}
Then the group corresponding to our Lie algebras $\g^{(1)}(N,\lambda,a,b)_+$ and $\g^{(3)}(N,\lambda,\mu)_+$ is:
$$\left(\mathbb{V}_{-\frac{1}{N}}\oplus\ldots \oplus \mathbb{V}_{-\frac{N-1}{N}}\right) \rtimes G_{FdB}.$$

Let us conclude this paragraph with the description of the Lie algebra associated to $\g^{(2)}(N,\lambda,\mu)$.

\begin{prop}
The Lie algebra $\g^{(2)}(N,\lambda,\mu)_+$ admits a decomposition 
$\g^{(2)}(N,\lambda,\mu)_+\approx V^{\oplus N} \rtimes \g_0$, where:
\begin{itemize}
\item $\g_0$ is an abelian, one-dimensional, Lie algebra, generated by an element $z$.
\item $V$ is a right $\g_0$-module, with a basis $(f_i)_{i\geq 0}$, and the right action defined by $f_i.z=f_{i+1}$.
\end{itemize} \end{prop}

\begin{proof} The Lie bracket of $\g^{(2)}(N,\lambda,\mu)_+$ is given by:
$$[X^i,X^j]=\begin{cases}
0\mbox{ if }i,j \neq N,
\mu i X^{i+N} \mbox{ if } i\neq N, j=N.
\end{cases}$$
We put $\g_0=Vect(X^N)$. The $N$-copies of $V$ are given by:
\begin{itemize}
\item For $1\leq r<N$, $\displaystyle V^{(r)}=Vect\left(\mu^i \prod_{j=1}^{i-1}(r+jN) X^{r+iN}\mid i\geq 0\right)$.
\item $V^{(N)}=Vect\left(\mu^iN^i(i+1)!X^{(i+2)N}\mid i\geq 0\right)$.
\end{itemize} \end{proof}

\subsection{Graded preLie products on $\K[X]$}

We now look for all  preLie products on $\K[X]$, making it a graded Com-PreLie algebra.
Let $\bullet$  be a such a product. By homogeneity, for all $i,j\geq 0$, there exists a scalar $\lambda_{i,j}$ such that:
$$X^i\bullet X^j=\lambda_{i,j}X^{i+j}.$$
Moreover, for all $i,j,k \geq 0$:
\begin{align*}
X^{i+j}\bullet X^k&=\lambda_{i+j,k}X^{i+j+k}\\
&=(X^iX^j)\bullet X^k\\
&=(X^i \bullet X^k)X^j+X^i(X^j \bullet X^k)\\
&=(\lambda_{i,k}+\lambda_{j,k})X^{i+j+k}.
\end{align*}
Hence, $\lambda_{i+j,k}=\lambda_{i,k}+\lambda_{j,k}$. Putting $\lambda_k=\lambda_{1,k}$ for all $k\geq 0$, we obtain:
$$X^i\bullet X^j=i\lambda_j X^{i+j}.$$

\begin{lemma} \label{21}
For all $k\geq 0$, let $\lambda_k \in \K$. We define a product $\bullet$ on $\K[X]$ by:
$$X^i\bullet X^j=i\lambda_j X^{i+j}.$$
Then $(\K[X],m,\bullet)$ is Com-PreLie if, and only if, for all $j,k\geq 1$:
$$(j\lambda_k-k\lambda_j)\lambda_{j+k}=(j-k)\lambda_j\lambda_k.$$
\end{lemma}

\begin{proof} Let $i,j,k\geq 0$. Then:
$$X^i\bullet(X^j \bullet X^k)-(X^i\bullet X^j)\bullet X^k=(ij \lambda_k \lambda_{j+k}-i(i+j)\lambda_j\lambda_k)X^{i+j+k}.$$
Hence:
\begin{align*}
\bullet\mbox{ is preLie}&\Longleftrightarrow
\forall i,j,k\geq 0, ij \lambda_k \lambda_{j+k}-i(i+j)\lambda_j\lambda_k=ik \lambda_j \lambda_{j+k}-i(i+k)\lambda_j\lambda_k\\
&\Longleftrightarrow \forall j,k\geq 0, (j\lambda_k-k\lambda_j)\lambda_{j+k}=(j-k)\lambda_j\lambda_k\\
&\Longleftrightarrow \forall j,k\geq 1, (j\lambda_k-k\lambda_j)\lambda_{j+k}=(j-k)\lambda_j\lambda_k,
\end{align*} 
as this relation is trivially satisfied if $j=0$ or $k=0$. \end{proof}

\begin{lemma}\label{22}
Let $\bullet$ be a product on $\K[X]$, making it a graded Com-PreLie algebra. Then $(\K[X],\prec,\bullet)$ is a Zinbiel-PreLie algebra.
\end{lemma}

\begin{proof} Let us take $i,k,k\geq 0$, $(i,j)\neq(0,0)$. Then:
\begin{align*}
(X^i\bullet X^k)\prec X^j+X^i \prec (X^j \bullet X^k)&=\lambda_k(i X^{i+k}\prec X^j+j X^i \prec X^{j+k})\\
&=\lambda_k\left(\frac{i(i+k)}{i+j+k}+\frac{ij}{i+j+k}\right)X^{i+j+k}\\
&=i\lambda_kX^{i+j+k}\\
&=(i+j)\lambda_k \frac{i}{i+j} X^{i+j+k},\\
(X^i\prec X^j)\bullet X^k&=\frac{i}{i+j}X^{i+j}\bullet X^k\\
&=\frac{i}{i+j}(i+j)\lambda_k X^{i+j+k}\\
&=i\lambda_k X^{i+j+k}.
\end{align*}
So $\K[X]$ is Zinbiel-PreLie. \end{proof}\\

\begin{proof} (Theorem \ref{18}, first part). Let us first prove that the objects defined in theorem \ref{18} are indeed Zinbiel-PreLie algebras.
By lemma \ref{22}, it is enough to prove that they are Com-PreLie algebras. We shall use lemma \ref{21} in all cases. \\

1. For all $j\geq 1$, $\lambda_j=a\frac{1}{\frac{j}{N}+b}$ if $N\mid j$ and $0$ otherwise. If $j$ or $k$ is not a multiple of $N$, then:
$$(j\lambda_k-k\lambda_j)\lambda_{j+k}=(j-k)\lambda_j\lambda_k=0.$$
If $j=Nj'$ and $k=Nk'$, with $j',k'$ integers, then:
\begin{align*}
(j\lambda_k-k\lambda_j)\lambda_{j+k}&=Na^2\left(\frac{j'}{k'+b}-\frac{k'}{j'+b}\right)\frac{1}{j'+k'+b}\\
&=Na^2\frac{j'^2-k'^2+b(j'-k')}{(j'+b)(k'+b)(j'+k'+b)}\\
&=Na^2(j'-k')\frac{j'+k'+b}{(j'+b)(k'+b)(j'+k'+b)}\\
&a^2(j-k)\frac{1}{(j'+b)(k'+b)}\\
&=(j-k)\lambda_j\lambda_k.
\end{align*}

2. In this case, $\lambda_j=\mu$ if $j=N$ and $0$ otherwise. Hence, for all $j,k \geq 1$:
\begin{align*}
(j\lambda_k-k\lambda_j)\lambda_{j+k}&=\mu^2(j\delta_{k,N}-k\delta_{j,N})\delta_{j+k,N}=0,\\
(j-k)\lambda_j\lambda_k&=\mu^2(j-k) \delta_{j,N} \delta_{k,N}=0.
\end{align*}

3. Here, for all $j\geq 1$, $\lambda_j=\mu$ if $N\mid j$ and $0$ otherwise. Then:
\begin{align*}
(j\lambda_k-k\lambda_j)\lambda_{j+k}&=\begin{cases}
\mu^2(j-k) \mbox{ if }N\mid j,k,\\
0\mbox{ otherwise};
\end{cases}&
(j-k)\lambda_j\lambda_k&=\begin{cases}
\mu^2(j-k) \mbox{ if }N\mid j,k,\\
0\mbox{ otherwise}.
\end{cases}\end{align*}

4. In this case, for all $j\geq 1$, $\lambda_j=0$ and the result is trivial. \end{proof}

\subsection{Classification of graded preLie products on $\K[X]$}

We now prove that the preceding examples cover all the possible cases.\\

\begin{proof} (Theorem \ref{18}, second part).  We put $X^i \bullet X^j=i\lambda_j X^{i+j}$ for all $i,j \geq 0$, and we put $\lambda=\lambda_0$. 
If for all $j\geq 1$, $\lambda_j=0$, then $\g=\g^{(4)}(\lambda)$. If this is not the case, we put:
$$N=\min\{j\geq 1\mid \lambda_j \neq 0\}.$$

{\it First step.} Let us prove that if $i$ is not a multiple of $N$, then $\lambda_i=0$.We put $i=qN+r$, with $0<r<N$, 
and we proceed by induction on $q$. By definition of $N$, $\lambda_1=\ldots=\lambda_{N-1}=0$, which is the result for $q=0$. 
Let us assume the result at rank $q-1$, with $q>0$.  We put $j=i-N$ and $k=N$. By the induction hypothesis, $\lambda_j=0$. Then, by lemma \ref{21}:
$$(i-N) \lambda_N\lambda_i=0.$$
As $i\neq N$ and $\lambda_N \neq 0$, $\lambda_i=0$. It is now enough to determine $\lambda_{iN}$ for all $i\geq 1$. \\

{\it Second step.} Let us assume that $\lambda_{2N}=0$. Let us prove that $\lambda_{iN}=0$ for all $i\geq 2$, by induction on $i$.
This is obvious if $i=2$. Let us assume the result at rank $i-1$, with $i\geq 3$, and let us prove it at rank $i$. 
We put $j=(i-1)N$ and $k=N$. By the induction hypothesis, $\lambda_j=0$. Then, by lemma \ref{21}:
$$(i-2)N\lambda_N \lambda_iN=0.$$
As $i\geq 3$ and $\lambda_N \neq 0$, $\lambda_{iN}=0$. As a conclusion, if $\lambda_{2N}=0$, putting $\mu=\lambda_N$, $\g=\g^{(2)}(N,\lambda,\mu)$. \\

{\it Third step.} We now assume that $\lambda_{2N} \neq 0$. We first prove that $\lambda_{iN}\neq 0$ for all $i\geq 1$.
This is obvious if $i=1,2$. he result at rank $i-1$, with $i\geq 3$, and let us prove it at rank $i$. 
We put $j=(i-1)N$ and $k=N$.  Then, by lemma \ref{21}:
$$(j\lambda_N-N\lambda_j)\lambda_{iN}=(i-2)N\lambda_j\lambda_N.$$
By the induction hypothesis, $\lambda_j\neq 0$. Moreover, $i>2$ and $\lambda_N\neq 0$, so $\lambda_{iN} \neq 0$.\\

For all $j\geq 1$, we put $\displaystyle \mu_j=\frac{\lambda_{kN}}{\lambda_N}$: this is a nonzero scalar, and $\mu_1=1$. 
Let us prove inductively that:
\begin{align*}
\mu_k&=\frac{\mu_2}{(k-1)-(k-2)\mu_2},&\mu_2&\neq \frac{k-1}{k-2} \mbox{ if }k\neq 2.
\end{align*}
If $k=1$, $\mu_1=1=\frac{\mu_2}{0-(-1)\mu_2}$, and $\mu_2\neq 0$ as $\lambda_{2N}\neq 0$; 
if $k=2$, $\mu_2=\frac{\mu_2}{1-0\mu_2}$. Let us assume the result at rank $k-1$, with $k\geq 3$. 
By lemma \ref{21}, with $j=(k-1)N$ and $k=N$:
\begin{align*}
((k-1)N \lambda_N-\lambda_N\mu_{k-1})\lambda_N\mu_k&=(k-2)N \mu_{k-1}\mu_1\lambda_N^2,\\
\mu_k(k-1-\mu_{k-1})&=(k-2)\mu_{k-1}.
\end{align*}
As $\mu_{k-1}\neq 0$ and $k>2$, $k-1-\mu_{k-1}\neq 0$. Moreover, by the induction hypothesis:
\begin{align*}
k-1-\mu_{k-1}&=k-1-\frac{\mu_2}{(k-2)-(k-3)\mu_2}\\
&=\frac{(k-1)(k-2)-((k-1)(k-3)+1)\mu_2}{(k-2)-(k-3)\mu_2}\\
&=(k-2)\frac{(k-1)-(k-2)\mu_2}{(k-2)-(k-3)\mu_2}.
\end{align*}
As this is nonzero, $\mu_2\neq \frac{k-1}{k-2}$. We finally obtain:
$$\mu_k=(k-2)\mu_{k-1}\frac{1}{k-2}\frac{(k-2)-(k-3)\mu_2}{(k-1)-(k-2)\mu_2}=\frac{\mu_2}{(k-1)-(k-2)\mu_2}.$$
Finally, for all $k\geq 1$:
$$\lambda_{kN}=\frac{\lambda_N\mu_2}{(k-1)-(k-2)\mu_2}=\frac{\lambda_N\mu_2}{(1-\mu_2)k+2\mu_2-1}.$$

{\it Last step.} If $\mu_2=1$, then for all $k\geq 1$, $\lambda_{kN}=\lambda_N$: this is $\g^{(3)}(N,\lambda,\lambda_N)$. If $\mu_2\neq 1$,
we put  $b=\frac{2\mu_2-1}{1-\mu_2}$. 
\begin{itemize}
\item As $\mu_2\neq 0$, $b\neq -1$;
\item $b\neq -2$;
\item for all $k\geq 3$, $\mu_2\neq \frac{k-1}{k-2}$, so $b\neq -k$. 
\end{itemize}
This gives that $b\notin \mathbb{Z}_-$. Moreover, for all $k\geq 1$:
$$\lambda_{kN}=\frac{\frac{\lambda_N\mu_2}{1-\mu_2}}{k+b}.$$
We take $a=\frac{\lambda_N\mu_2}{1-\mu_2}$, and we obtain $\g^{(1)}(N,\lambda,a,b)$. \end{proof}

\begin{prop}\label{23}
Among the examples of theorem \ref{18}, the Com-PreLie bialgebras (or equivalently the Zinbiel-PreLie bialgebras) are 
$\g^{(1)}(1,a,1)$ for all $a\neq 0$ and $g^{(4)}(0)$. 
\end{prop}

\begin{proof}  Note that $\g^{(1)}(1,0,1)=\g^{(4)}(0)$. 
Let us first prove that $\g(1,a,1)$ is a Zinbiel-PreLie bialgebra for all $a\in \K$.
Let us take $V$ one-dimensional, generated by $x$, with $f=a Id$. We work in $T(V,f)$.
Let us prove that $x^k \bullet x^l=a\binom{k+l}{k-1}x^{k+l}$ by induction on $k$. It is obvious if $k=0$, as $\binom{k+l}{-1}=0$.
Let us assume the result at rank $k-1$.
\begin{align*}
x^k \bullet x^l&=x(x^{k-1}\shuffle x^l)+f(x)x^{k-1}\shuffle x^l\\
&=a\left(\binom{k+l-1}{k-1}+\binom{k+l-1}{k-1}\right)x^{k+l}\\
&=a\binom{k+l}{k-1}x^{k+l}.
\end{align*}
The Zinbiel product of $T(V)$ is given by:
$$x^k \prec x^l=\binom{k+l-1}{k-1} x^{k+l},$$
for all $k,l\geq 1$. There is an isomorphism of Hopf algebras:
$$\Theta:\left\{\begin{array}{rcl}
\K[X]&\longrightarrow&T(V)\\
X&\longrightarrow&x.
\end{array}\right.$$
For all $n\geq 0$, $\Theta(X^n)=x^{\shuffle n}=n!x^n$. For all $k,l\geq 0$:
\begin{align*}
\Theta(X^k)\bullet \Theta(X^l)&=a \binom{k+l}{k-1} k!l! x^{k+l}&\Theta(X^k)\prec \Theta(X^l)&=\binom{k+l-1}{k-1}k!l!x^{k+l}\\
&=\frac{a k}{l+1}\Theta(X^{k+l});&&=\frac{k}{k+l}\Theta(X^{k+l}).
\end{align*}
Consequently, $\g^{(1)}(1,a,1)$  is isomorphic, as a Zinbiel-PreLie bialgebra to $T(V,f)$ (so is indeed a Zinbiel-PreLie bialgebra).\\

Let $\g$ be one of the examples of theorem \ref{18}.  First:
\begin{align*}
\Delta(X\bullet X)&=X\otimes 1\bullet X+1\otimes X\bullet X\\
&+X\bullet X\otimes 1+X\bullet 1\otimes X+1\bullet X\otimes X+1\bullet 1\otimes X^2\\
\lambda_1(1\otimes X^2+2X\otimes X+X^2\otimes 1)&=\lambda_11\otimes X^2+\lambda X \otimes X+\lambda_1 X^2\otimes 1.
\end{align*}
This gives $\lambda=2\lambda_1$. In particular, if $\g=\g^{(4)}(\lambda)$, then $\lambda=2\lambda_1=0$: this is $\g^{(4)}(0)$.
In the other cases, $N$ exists. By definition of $N$, $X\bullet X^k=0$ if $1\leq k \leq N-1$. We obtain:
\begin{align*}
\Delta(X\bullet X^N)&=1\otimes X\bullet X^N+X \otimes 1\bullet X^N
+\sum_{k=0}^N \binom{N}{k} (X\bullet X^k \otimes X^{N-k}+1\bullet X^k \otimes X^{n-k+1})\\
\lambda_N \Delta(X^{N+1})&=1\otimes X\bullet X^N+\lambda X \otimes X^N+1\otimes X\bullet X^N.
\end{align*}
If $\lambda=0$, we obtain that $X^{N+1}$ is primitive, so $N+1=1$: absurd, $N\geq 1$.
So $\lambda \neq 0$. The cocommutativity of $\Delta$ implies that $N=1$.
\begin{align*}
\Delta(X\bullet X^2)&=\lambda_2(X^3\otimes 1+3X^2\otimes X+3X\otimes X^2+1\otimes X^3)\\
&=1\otimes X\bullet X^2+2\lambda_1 X^2\otimes X+\lambda_0X\otimes X^2+1\otimes X\bullet X^2 
\end{align*}
Hence, $3\lambda_2=2\lambda_1$. 
\begin{itemize}
\item If $\g=\g^{(3)}(1,\lambda,\mu)$, we obtain $3\mu=2\mu$, so $\mu=0$: this is a contradiction.
\item If $\g=\g^{(2)}(1,\lambda,\mu)$, we obtain $0=2\mu$, so $\mu=0$: this is a contradiction.
\end{itemize}
So $\g=\g^{(1)}(1,\lambda,a,b)$. We obtain:
$$3\frac{a}{2+b}=2\frac{a}{1+b},$$
so $b=1$. Then $\lambda_0=2\lambda_1=\frac{2a}{2}=a=\frac{a}{b}$, so $\g=\g^{(1)}(1,a,1)$.\end{proof}

\section{Cocommutative Com-PreLie bialgebras}

We shall prove the following theorem:

\begin{theo}\label{24}
Let $A$ be a connected, cocommutative Com-PreLie bialgebra. Then one of the following assertions holds:
\begin{enumerate}
\item There exists a linear form $f:V\longrightarrow \K$ and $\lambda \in \K$, such that $A$ is isomorphic to $S(V,f,\lambda)$.
\item There exists $\lambda \in \K$ such that $A$ is isomorphic to $\g^{(1)}(1,\lambda,1)$.
\end{enumerate}\end{theo}

First, observe that if $A$ is a cocommutative, commutative, connected Hopf algebra: by the Cartier-Quillen-Milnor-Moore theorem,
it is isomorphic to the enveloping Hopf algebra of an abelian Lie algebra, so is isomorphic to $S(V)$ as a Hopf algebra,
where $V=Prim(A)$. If $V=(0)$, the first point holds trivially. 

\subsection{First case}

We assume in this paragraph that $V$ is at least $2$-dimensional.

\begin{lemma}
Let $A$ be a connected, cocommutative Com-PreLie algebra, such that the dimension of $Prim(A)$ is at least $2$. Then $f_A=0$,
and there exists a map $F:A\otimes A\longrightarrow A$, such that:
\begin{enumerate}
\item For all $x,y \in A_+$, $x\bullet y=F(x\otimes y')y''+F(x\otimes 1)y$.
\item For all $x_1,x_2 \in A$, $F(x_1x_2\otimes y)=F(x_1\otimes y)x_2+x_1F(x_2\otimes y)$.
\item $F(Prim(A)\otimes A)\subseteq \K$.
\end{enumerate}\end{lemma}

\begin{proof} We assume that $A=S(V)$ as a bialgebra, with its usual product and coproduct $\bfDelta$, and that $dim(V)\geq 2$. Let $x,y\in V$. Then:
$$\bfDelta(x\bullet y)=x\bullet y\otimes 1+1\otimes x\bullet y+f_A(x)\otimes y.$$
By cocommutativity, for all $x,y\in V$, $f_A(x)$ and $y$ are colinear. Let us choose $y_1$ and $y_2 \in V$, non colinear.
Then $f_A(x)$ is colinear to $y_1$ and $y_2$, so belongs to $Vect(y_1)\cap Vect(y_2)=(0)$. Finally, $f_A=0$. \\

We now construct linear maps $F_i:V\otimes S^i(V)\longrightarrow \K$, such that for all $k\geq 0$, putting:
$$F^{(k)}=\displaystyle \bigoplus_{i=0}^k F_i:\bigoplus_{i=0}^k V\otimes S^i(V)\longrightarrow \K,$$
for all $x,y_1,\ldots,y_{k+1}\in V$:
$$x\bullet y_1\ldots y_{k+1}=F^{(k)}(x\otimes (y_1\ldots y_{k+1})')\otimes (y_1\ldots y_{k+1})''+F^{(k)}(x\otimes 1)y_1\ldots y_{k+1}.$$ 
We proceed by induction on $k$. Let us first construct $F^{(0)}$. Let $x,y\in V$.
$$\bfDelta(x\bullet y^2)=1\otimes x\bullet y^2+x\bullet y^2\otimes 1+2x\bullet y\otimes y.$$
By cocommutativity, $x\bullet y$ and $y$ are colinear, so there exists a linear map $g:V\longrightarrow \K$
such that $x\bullet y=g(x)y$. We the take $F^{(0)}(x\otimes 1)=g(x)$. For all $x,y\in V$, $x\bullet y=F(x\otimes 1)y$,
so the result holds for $k=0$. \\

Let us assume that $F^{(0)},\ldots,F^{(k-2)}$ are constructed for $k\geq 2$. Let $x,y_1,\ldots,y_k\in V$.
For all $I\subseteq [k]=\{1,\ldots,k\}$, we put $\displaystyle y_I=\prod_{i\in I} y_i$. Then:
$$\bftDelta(y_1\ldots y_k)=\sum_{I\sqcup J=[k], I,J\neq 1} y_I\otimes y_J,$$
and:
\begin{align*}
\bfDelta(x\bullet y_1\ldots y_k)&=1\otimes x\bullet y_1\ldots y_k+x\bullet y_1\ldots y_k\otimes 1+\sum_{[k]=I\sqcup J,J\neq 1} x\bullet y_I\otimes y_J\\
&=1\otimes x\bullet y_1\ldots y_k+x\bullet y_1\ldots y_k\otimes 1+\sum_{I\sqcup J\sqcup K=[k], J,K\neq 1}
F^{(k-2)}(x \otimes y_I)\otimes y_J \otimes y_K. 
\end{align*}
 We put:
 $$P(x,y_1\ldots y_k)=x\bullet y_1\ldots y_k-\sum_{I\sqcup J=[k], |J|\geq 2} F^{(k-2)}(x\otimes y_I)y_J.$$
 The preceding computation shows that $P(x,y_1\ldots,y_k)$ is primitive, so belongs to $V$. Let $y_{k+1}\in V$. 
\begin{align*}
\bftDelta(x\bullet y_1\ldots y_{k+1})&=\sum_{I\sqcup J\sqcup K=[k+1],K\neq 1,|J|\geq 2}
\underbrace{F^{(k-2)}(x\otimes y_I)y_J}_{\in S_{\geq 2}(V)}\otimes y_K\\
&+P(x,y_1\ldots y_k)\otimes y_{k+1}+\sum_{i=1}^kP(x, y_1\ldots y_{i-1}y_{i+1}\ldots y_{k+1})\otimes y_i.
\end{align*}
By cocommutativity, considering the projection on $V\otimes V$, we deduce that $P(x, y_1\ldots y_k)\in Vect(y_1,\ldots, y_k,y_{k+1})$
for all nonzero $y_{k+1}\in V$. In particular, for $y_1=y_{k+1}$, $P(x\otimes y_1\ldots y_k)\in Vect(y_1,\ldots,y_k)$. 
By multilinearity, there exists $F'_1,\ldots,F'_k\in (V\otimes S_{k-1}(V))^*$, such that for all $x,y_1,\ldots,y_k \in V$:
$$P(x,y_1\ldots y_k)=F'_1(x\otimes y_2\ldots y_k)y_1+\ldots+F'_k(x\otimes y_1\ldots y_{k-1})y_k.$$
By symmetry in $y_1,\ldots,y_k$, $F'_1=\ldots=F'_k=F_{k-1}$. Then:
\begin{align*}
x\bullet y_1\ldots y_k&=\sum_{I\sqcup J=[k],|J|\geq 2} F^{(k-2)}(x\otimes y_I)y_J+\sum_{I\sqcup J=[k],|J|=1}F_{k-1}(x\otimes y_I)y_J\\
&=\sum_{I\sqcup J=[k],|J|\geq 1} F^{(k-1)}(x\otimes y_I)y_J\\
&=F^{(k-1)}(x\otimes (y_1\ldots y_k)')(y_1\ldots y_k)''+F(x\otimes 1)y_1\ldots y_k.
\end{align*}

We defined a map $F:V\otimes S(V)\longrightarrow K$, such that for all $x\in V$, $b\in S_+(V)$,
$$x\bullet b=F(x\otimes b')b''+F(x\otimes 1)b.$$
We extend $F$ in a map from $S(V)\otimes S(V)$ to $S(V)$ by:
\begin{itemize}
\item  $F(1\otimes b)=0$.
\item For all $x_1,\ldots,x_k \in V$, $\displaystyle F(x_1\ldots x_k\otimes b)=\sum_{i=1}^k x_1\ldots x_{i-1}F(x_1\otimes b)x_{i+1}\ldots x_k$.
\end{itemize}
This map $F$ satisfies points 2 and 3. Let us consider:
$$B=\{a\in A\mid \forall  b\in S_+(V), a\bullet b=F(a\otimes b')b''+F(a\otimes 1)b\}.$$
As $1\bullet b=0$ for all $b\in S(V)$, $1\in B$. By construction of $F$, $V\subseteq B$.
Let $a_1,a_2\in B$. For any $b\in S_+(V)$:
\begin{align*}
a_1a_2\bullet b&=(a_1\bullet b)a_2+a_1(a_2\bullet b)\\
&=F(a_1\otimes b')a_2b''+a_1F(a_2\bullet b')b''+F(a_1\otimes 1)a_2b+a_1F(a_2\otimes 1)b\\
&=F(a_1a_2\otimes b')b''+F(a_1a_2\otimes 1)b. 
\end{align*}
So $a_1a_2\in B$. Hence, $B$ is a subalgebra of $S(V)$ containing $V$, so is equal to $S(V)$: $F$ satisfies the first point. \end{proof}\\

{\bf Remarks.} \begin{enumerate}
\item In this case, for all primitive element $v$, the $1$-cocycle of the bialgebra $A$ defined by $L(x)=a\bullet x$
is the coboundary associated to the linear form defined by $f(x)=-F(a\otimes x)$
\item In particular, the preLie product of two elements $x,y$ of $Prim(A)$ si given by:
$$x\bullet y=F(x\otimes 1)y.$$
\end{enumerate}

\begin{lemma}\label{26}
With the preceding hypothesis, let us assume that $F(x\otimes 1)=0$ for all $x\in Prim(A)$. Then $\bullet=0$.
\end{lemma}

\begin{proof} We assume that $A=S(V)$ as a bialgebra. By hypothesis, for all $a\in A$, $F(a\otimes 1)=0$, so $a\bullet 1=0$.
This implies that for all $a,b\in S_+(V)$:
$$\bftDelta(a\bullet b)=a\bullet b'\otimes b''+a'\bullet b'\otimes a''b''+a'\bullet b\otimes a''+a'\otimes a''\bullet b.$$
Let us prove the following assertion by induction on $N$: for all $k<N$, for all $x,y_1,\ldots,y_k \in V$,
$x\bullet y_1\ldots y_k=0$. 
By hypothesis, this is true for $N=1$. Let us assume the result at a certain rank $N\geq 2$.
Let us choose $x,y_1,\ldots,y_N\in V$. Then, by the induction hypothesis:
$$\bftDelta(x\bullet y_1\ldots y_N)=0+0+0+0=0.$$
So $x\bullet y_1\ldots y_N$ is primitive. 

Up to a factorization, we can write any $x \bullet y_1\ldots y_N$ as a linear span of terms of the form
$z_1\bullet z_1^{\beta_1}\ldots z_n^{\beta_n}$, with $z_1,\ldots,z_n$ linearly independent, $\beta_1,\ldots,\beta_n\in \N$,
with $\beta_1+\ldots+\beta_n=N$. If $n=1$, as $dim(V)\geq 2$ we can choose any $z_2$ linearly independent with $z_1$ and take $\beta_2=0$.
It is now enough to consider $z_1\bullet z_1^{\beta_1}\ldots z_n^{\beta_n}$, with $n \geq 2$, $z_1,\ldots,z_n$ linearly independent, 
$\beta_1,\ldots,\beta_n\in \N$, $\beta_1+\ldots+\beta_n=N$. Let $\alpha_1,\ldots,\alpha_n\in \N$, such that $\alpha_1+\ldots+\alpha_n=N+1$.
\begin{align*}
 \bftDelta(z_1\bullet z_1^{\alpha_1}\ldots z_n^{\alpha_n})&=\sum_{i=1}^n \alpha_i z_1\bullet z_1^{\alpha_1}\ldots z_i^{\alpha_i-1}\ldots
 z_n^{\alpha_n}\otimes z_i,\\ \\
\bftDelta\left(\frac{z_1^2}{^2}\bullet z_1^{\alpha_1}\ldots z_n^{\alpha_n}\right)&=
\sum_{i=1}^n \alpha_i (z_1\bullet z_1^{\alpha_1}\ldots z_i^{\alpha_i-1}\ldots z_n^{\alpha_n})z_1\otimes z_i\\
&+\sum_{i=1}^n \alpha_i z_1\bullet z_1^{\alpha_1}\ldots z_i^{\alpha_i-1}\ldots z_n^{\alpha_n}\otimes z_1z_i\\
&+z_1\bullet z_1^{\alpha_1}\ldots z_n^{\alpha_n}\otimes z_1+z_1\otimes z_1\bullet z_1^{\alpha_1}\ldots z_n^{\alpha_n},\\ \\
(\bftDelta \otimes Id)\circ \bftDelta\left(\frac{z_1^2}{^2}\bullet z_1^{\alpha_1}\ldots z_n^{\alpha_n}\right)&=
\sum_{i=1}^n \alpha_i z_1\bullet z_1^{\alpha_1}\ldots z_i^{\alpha_i-1}\ldots z_n^{\alpha_n}\otimes z_1 \otimes z_i\\
&+\sum_{i=1}^n \alpha_i z_1\otimes z_1\bullet z_1^{\alpha_1}\ldots z_i^{\alpha_i-1}\ldots z_n^{\alpha_n}\otimes z_i\\
&+\sum_i \alpha_iz_1\bullet z_1^{\alpha_1}\ldots z_i^{\alpha_i-1}\ldots z_n^{\alpha_n}\otimes z_i\otimes z_1.
\end{align*}
The cocommutativity implies that for all $1\leq i\leq n$, $\alpha_i z_1\bullet z_1^{\alpha_1}\ldots z_i^{\alpha_i-1}\ldots z_n^{\alpha_n}$
and $z_i$ are colinear. We first choose $\alpha_1=\beta_1+1$, $\alpha_i=\beta_i$ for all $i\geq 2$, and we obtain for $i=1$ that
$z_1\bullet z_1^{\beta_1}\ldots z_n^{\beta_n} \in Vect(z_1)$. We then choose $\alpha_n=\beta_n+1$ and $\alpha_i=\beta_i$ for all $i\leq n-1$,
and we obtain for $i=n$ that $z_1\bullet z_1^{\beta_1}\ldots z_n^{\beta_n} \in Vect(z_n)$. Finally, as $n\geq 2$,
$z_1\bullet z_1^{\beta_1}\ldots z_n^{\beta_n} \in Vect(z_1)\cap Vec(z_2)=(0)$; the hypothesis is true at trank $N$.\\

We proved that for all $x\in V$, for all $b\in S(V)$, $x\bullet b=0$. By the derivation property of $\bullet$, as $V$ generates $S(V)$,
for all $a,b \in S(V)$, $a\bullet b=0$. \end{proof}

\begin{lemma} 
Under the preceding hypothesis, Let us assume that $F(Prim(A)\otimes \K)\neq (0)$. Then $A$ is isomorphic to a certain $S(V,f,\lambda)$,
with $V=Prim(A)$ and $f(x)=F(x\otimes 1)$ for all $x\in V$.
\end{lemma}

\begin{proof}  We assume that $A=S(V)$ as a bialgebra. Let $a,b,c\in S_+(V)$. Then:
\begin{align*}
\bftDelta([a,b])&=a'\otimes a''\bullet b+a\bullet b'\otimes b''+a'\bullet b\otimes a''\\
&-b'\otimes b''\bullet a-b\bullet a'\otimes a''-b'\otimes a\otimes b''+[a',b']\otimes a''b'',
\end{align*}
where $[-,-]$ is the Lie bracket associated to $\bullet$. Hence:
\begin{align*}
(a\bullet b)\bullet c&=F(a\otimes 1)b\bullet c+F(a\otimes b')b''\bullet c\\
&=F(a\otimes 1)F(b\otimes 1)c+F(a\otimes 1)F(b\otimes c')c''\\
&+F(a\otimes b')F(b''\otimes 1)c+F(a\otimes b')F(b''\otimes c')c'',\\ \\
(a\bullet c)\bullet b&=F(a\otimes 1)F(c\otimes 1)b+F(a\otimes 1)F(c\otimes b')b''\\
&+F(a\otimes c')F(c''\otimes 1)b+F(a\otimes c')F(c''\otimes b')b'',\\ \\
a\bullet [b,c]&=F(a\otimes 1)F(b\otimes 1)c+F(a\otimes 1)F(b\otimes c')c''-F(a\otimes 1)F(c\otimes 1)b\\
&-F(a\otimes 1)F(c\otimes b')b''+F(a\otimes b')F(b''\otimes 1)c+F(a\otimes b')F(b''\otimes c')c''\\
&-F(a\otimes c')F(c''\otimes 1)b-F(a\otimes c')F(c''\otimes b')b''+F(a\otimes F(b\otimes 1)c')c''\\
&+F(a\otimes F(b\otimes c')c'')c'''-F(a\otimes F(c\otimes 1)b')b''-F(a\otimes F(c\otimes b')b'')b'''\\
&+F(a\otimes F(b'\otimes 1)c)b''+F(a\otimes F(b'\otimes c')c'')b''-F(a\otimes F(c'\otimes 1)b)c''\\
&+F(a\otimes F(c'\otimes b')b'')c''+F(a\otimes F(b'\otimes 1)c')b''c''+F(a\otimes F(b'\otimes c')c'')b''c'''\\
&-F(a\otimes F(c'\otimes 1)b')b''c'''-F(a\otimes F(c'\otimes b')b'')b'''c''.
\end{align*}
The preLie relation implies that:
\begin{align*}
0&=F(a\otimes F(b\otimes 1)c')c''+F(a\otimes F(b\otimes c')c'')c'''-F(a\otimes F(c\otimes 1)b')b''\\
&-F(a\otimes F(c\otimes b')b'')b'''+F(a\otimes F(b'\otimes 1)c)b''+F(a\otimes F(b'\otimes c')c'')b''\\
&-F(a\otimes F(c'\otimes 1)b)c''+F(a\otimes F(c'\otimes b')b'')c''+F(a\otimes F(b'\otimes 1)c')b''c''\\
&+F(a\otimes F(b'\otimes c')c'')b''c'''-F(a\otimes F(c'\otimes 1)b')b''c'''-F(a\otimes F(c'\otimes b')b'')b'''c''.
\end{align*}
For $a=x\in V$, $b=y\in V$, as $F(V\otimes S(V))\subset \K$, this simplifies to:
\begin{equation}
\label{EQ1} F(x\otimes c')F(y\otimes 1)c''+F(y\otimes c')F(x\otimes c'')c'''=F(x\otimes F(c'\otimes 1)y)c''.
\end{equation}
Let $x_1,\ldots,x_k \in V$, linearly independent, $\alpha_1,\ldots,\alpha_k\in \N$,
with $\alpha_1+\ldots+\alpha_N\geq 1$. We take $c=x_1^{\alpha_1+1}\ldots x_k^{\alpha_k}$
and $d=x_1^{\alpha_1}\ldots x_k^{\alpha_k}$.
The coefficient of $x_1$ in (\ref{EQ1}), seen as an equality between two polynomials in $x_1,\ldots,x_k$, gives:
$$(\alpha_1+1)(F(x\otimes d)F(y\otimes 1)+F(y\otimes d')F(x\otimes d''))=(\alpha_1+1)F(x\otimes F(d\otimes 1)y).$$
Hence, for all $x,y\in V$, for all $c\in S_+(V)$:
\begin{equation}
\label{EQ2} F(x\otimes c)F(y\otimes 1)+F(y\otimes c')F(x\otimes c'')=F(x\otimes F(c\otimes 1)y).
\end{equation}
We put $f(x)=F(x\otimes 1)$ for all $x\in V$. If $f=0$, by lemma \ref{26}, $\bullet=0$, so $A$ is isomorphic to $S(V,0,0)$.
Let us assume that $f\neq 0$ and let us choose $y\in V$, such that $f(y)=1$. If $z_1,\ldots,z_k \in Ker(f)$, then:
$$F(z_1\ldots z_k\otimes 1)=\sum_{i=1}^k z_1\ldots g(z_i)\ldots z_k=0.$$
Consequenlty, if $c\in S_+(Ker(f))\subseteq S_+(V)$, (\ref{EQ2}) gives:
$$F(x\otimes c')+F(y\otimes c')F(x\otimes c'')=0.$$
An easy induction on the length of $c$ proved that for all $c\in S_+(Ker(g))$, $F(x\otimes c)=0$ for all $x\in V$. 
So there exists linear forms $g_k \in V^*$, such that for all $x,y_1,\ldots,y_k \in V$:
$$F(x\otimes y_1\ldots y_k)=g_k(x)f(y_1)\ldots f(y_k).$$
In particular, $g_0=f$. The preLie product is then given by:
$$x\bullet y_1\ldots y_k=\sum_{i=1}^{k-1}g_i(x)\sum_{1\leq j_1<\ldots<j_i\leq k} y_1\ldots f(y_{j_1})\ldots f(y_{j_i})\ldots y_k.$$

Let $x,y,z_1,\ldots,z_k\in V$.
\begin{align*}
x\bullet (y\bullet z_1\ldots z_k)&=x\bullet \sum_{i=0}^{k-1}g_i(y)\sum_{j_1,\ldots,j_i} z_i\ldots f(z_{j-1})\ldots f(z_{j_i})\ldots z_l\\
&=\sum_{i=0}^{k-1}g_{l-i-1}(x) g_i(x)\binom{l-1}{i}\sum_{j=1}^k f(z_1)\ldots f(z_{j-1})z_j f(z_{j+1})\ldots f(z_k)+S_{\geq 2}(V),\\
(x\bullet y)\bullet z_1\ldots z_k&=f(x)y\bullet z_1\ldots z_k\\
&=f(x)g_{k-1}(y)\sum_{j=1}^k f(z_1)\ldots f(z_{j-1})z_j f(z_{j+1})\ldots f(z_k)+S_{\geq 2}(V),\\
x\bullet (z_1\ldots z_k\bullet y)&=\sum_{i=1}^k f(y_i)x\bullet z_1\ldots z_{i-1}z_{i+1}\ldots z_ky\\
&=k g_{k-1}(x)f(z_1)\ldots f(z_k)y\\
&+(k-1)f(x)g_{k-1}(y)\sum_{j=1}^k f(z_1)\ldots f(z_{j-1})z_j f(z_{j+1})\ldots f(z_k)+S_{\geq 2}(V),\\
(x\bullet z_1\ldots z_k)\bullet y&=\sum_{i=0}^{k-1} g_i(x)\sum_{j_1,\ldots,j_i} z_1\ldots f(z_{j_1})\ldots f(z_{j_i})\ldots z_k \bullet y\\
&=kf(x)g_{k-1}(y)\sum_{j=1}^k f(z_1)\ldots f(z_{j-1})z_j f(z_{j+1})\ldots f(z_k)+S_{\geq 2}(V).
\end{align*}
Let us choose $z_1=\ldots=z_k=z$, such that $f(z)=1$. Then:
$$\sum_{j=1}^k f(z_1)\ldots f(z_{j-1})z_j f(z_{j+1})\ldots f(z_k)=kz\neq 0.$$
The preLie relation implies:
$$f(x)g_{k-1}(y)+(k-1)g_{k-1}(x)f(y)-\sum_{i=0}^{k-1}g_i(y)g_{k-i-1}(x)\binom{k-1}{i}=0,$$
so, for all $l\geq 1$:
\begin{equation}
\label{EQ3} lg_l(x)f(y)=\sum_{i=1}^l g_i(y)g_{l_i}(x)\binom{l}{i}.
\end{equation}
Let us choose $x$ such that $f(x)=1$. Let us consider $y \in Ker(f)$, and let us prove that $g_i(y)=0$ for all $i\geq 0$.
As $g_0=f$, this is obvious for $i=0$. Let us assume the result at all rank $<l$, with $l\geq 1$. Then (\ref{EQ3}) gives:
$$0=\sum_{i=1}^{l-1} g_i(y)g_{l_i}(x)\binom{l}{i}+g_l(y)f(x)=g_l(y).$$
Consequently, for all $l\geq 1$, there exists a scalar $\lambda_l$ such that $g_l=\lambda_l f$. 
If $f(x)=f(y)=1$, equation (\ref{EQ3}) gives, for all $l\geq 1$:
$$l\lambda_l=\sum_{i=1}^l \lambda_i \lambda_{l-i}\binom{l}{i}=\sum_{i=1}^{l-1}\lambda_i\lambda_{l-i}\binom{l}{i}+\lambda_l,$$
so, for all $l\geq 2$:
$$\lambda_l=\frac{1}{l-1}\sum_{i=1}^{l-1}\lambda_i\lambda_{l-i}\binom{l}{i}.$$
An induction proves that $\lambda_l=l! \lambda_1^l$ for all $l\geq 1$. Putting $\lambda_1=\lambda$, for all $x,x_1,\ldots,x_n \in V$:
$$x\bullet x_1\ldots x_k=\sum_{I\subsetneq \{1,\ldots,k\}} |I|! \lambda^{|I|} f(x) \prod_{i\in I} f(x_i) \prod_{i\notin I} x_i.$$
This is the preLie product of $S(V,f,\lambda)$. \end{proof}

\subsection{Second case}

We now assume that $V$ is one-dimensional. So $S(V)$ and $\K[X]$ are isomorphic as bialgebras.
Let us describe all the preLie products on $\K[X]$ making it a Com-PreLie bialgebra.

\begin{prop}
Let $\lambda,\mu \in \K$. We define:
$$X^k \bullet X^l=\lambda k l! \sum_{i=k}^{k+l-1} \frac{\mu^{k+l-i-1}}{(i-k+1)!}X^i.$$
Then $(\K[X],m,\prec,\Delta)$ is a Zinbiel-PreLie algebra denoted by $\g'(\lambda,\mu)$.
\end{prop}

\begin{proof}  If $\lambda=0$, $\bullet=0$ and the result is obvious. Let us assume that $\lambda \neq 0$.
Let $V$ one-dimensional, $x\in V$, nonzero, and let $f\in V^*$ defined by $f(x)=\frac{\mu}{\lambda}$. In $T(V,f,\lambda)$, by lemma \ref{8},
for all $k,l\geq 0$:
$$x^k \bullet x^l=\lambda \sum_{i=k}^{k+l-1} \mu^{k+l-i-1} \binom{i}{k-1}x^j.$$
Let us consider the Hopf algebra isomorphism:
$$\Theta:\left\{\begin{array}{rcl}
\K[X]&\longrightarrow&T(V)\\
X&\longrightarrow&x.
\end{array}\right.$$
For all $k,l \geq 0$:
\begin{align*}
\Theta(X^k)\bullet \Theta(X^l)&=\lambda \sum_{i=k}^{k+l-1}\mu^{k+l-i-1}\frac{i!k!l!}{(k-1)!(i-k+1)!}x^i\\
&=\lambda k l! \sum_{i=k}^{k+l-1}\frac{\mu^{k+l-1-i}}{(i-k+1)!}\Theta(X^i).
\end{align*}
By proposition \ref{13}, $T(V,f,\lambda)$ is a Zinbiel-PreLie bialgebra, so is $\g'(\lambda,\mu)$. \end{proof}

\begin{prop}
Let $\bullet$ a preLie product on $\K[X]$ such that $(\K[X],m,\bullet,\Delta)$ is a Com-PreLie bialgebra. Then $(\K[X],m,\bullet,\Delta)=\g^{(1)}(1,\lambda,1)$
for a certain $\lambda \in \K$, or $\g'(\lambda,\mu)$ for a certain $(\lambda,\mu)\in \K^2$.
\end{prop}

\begin{proof} Let $\pi:\K[X]\longrightarrow \K[X]$ be the canonical projection on $Vect(X)$:
$$\pi:\left\{\begin{array}{rcl}
\K[X]&\longrightarrow  \K[X]\\
X^k&\longrightarrow \delta_{k,1}X.
\end{array}\right.$$
For all $k\geq 0$, we put $\pi(X\bullet X^k)=\lambda_k X$.\\

We shall use the map $\varpi=m\circ(\pi \otimes Id) \circ \Delta$. For all $k\geq 0$:
$$\varpi(X^k)=m\circ(\pi \otimes Id)\left(\sum_{i=0}^k \binom{k}{i}X^i \otimes X^{k-i}\right)=m(k X\otimes X^{k-1})=kX^k.$$

{\it First step.} We fix $l\geq 0$. For all $P,Q \in \K[X]$, $\varepsilon(P\bullet Q)=0$; hence, we can write:
$$X\bullet X^l=\sum_{i=1}^\infty a_iX^i.$$
Then:
\begin{align*}
\varpi(X\bullet X^l)&=\sum_{i=1}^\infty ia_iX^i\\
&=m\circ (\pi \otimes Id)\circ \Delta(X\bullet X^l)\\
&=m\circ (\pi \otimes Id)\left(1\otimes X\bullet X^l+\sum_{i=0}^l \binom{l}{i}X\bullet X^i\otimes X^{l-i}\right)\\
&=m\left(\sum_{i=0}^l \binom{l}{i}\lambda_i X\otimes X^{l-i}\right)\\
&=\sum_{i=0}^l \binom{l}{i}\lambda_i  X^{l-i+1}\\
&=\sum_{j=1}^{l+1}\binom{l}{l-j+1}\lambda_{l-j+1}X^j.
\end{align*}
Hence:
$$X\bullet X^l=\sum_{j=1}^{l+1}\binom{l}{l-j+1}\frac{\lambda_{l-j+1}}{j}X^j.$$
By derivation, for all $k\geq 0$, $X^k\bullet X^l=kX^{k-1}(X\bullet X^l)$, so for all $k,l\geq 0$:
$$X^k \bullet X^l=\sum_{j=1}^{l+1} k\binom{l}{l-j+1} \frac{\lambda_{l-j+1}}{j} X^{j+k-1}.$$

{\it Second step.} In particular, for all $k\geq 0$, $X^k\bullet 1=k\lambda_0 X^k$, and $X\bullet X=\frac{\lambda_0}{2}X^2+\lambda_1X$. Hence:
\begin{align*}
X\bullet (X\bullet 1)-(X\bullet X)\bullet 1&=\frac{\lambda_0^2}{2}X^2+\lambda_0\lambda_1X-\frac{\lambda_0}{2}X^2\bullet 1-\lambda_1X\bullet 1\\
&=\frac{\lambda_0^2}{2}X^2+\lambda_0\lambda_1X-\lambda_0^2X^2-\lambda_0\lambda_1X\\
&=-\frac{\lambda_0^2}{2}X^2;\\ \\
X\bullet (1\bullet X)-(X\bullet 1)\bullet X&=0-\lambda_0X\bullet X\\
&=-\frac{\lambda_0^2}{2}X^2-\lambda_0\lambda_1X.
\end{align*}
By the preLie relation, $\lambda_0\lambda_1=0$. We shall now study three cases:
\begin{align*}
\mathbf{1.}\:& \begin{cases}
\lambda_0&\neq 0,\\
\lambda_1&=0;
\end{cases}&
\mathbf{2.}\:& \begin{cases}
\lambda_0&=0,\\
\lambda_1&=0;
\end{cases}&
\mathbf{3.}\:& \begin{cases}
\lambda_0&=0,\\
\lambda_1&\neq 0;
\end{cases}&
\end{align*}

{\it Third step.} First case: $\lambda_0\neq 0$, $\lambda_1=0$. 
Let us prove that $\lambda_k=0$ for all $k\geq 1$ by induction on $k$. It is obvious if $k=1$.
Let us assume that $\lambda_1=\ldots=\lambda_{k-1}=0$.  Then $X\bullet X^k=\frac{\lambda_0}{k+1}X^{k+1}+\lambda_k X$, and:
\begin{align*}
X\bullet (X^k \bullet 1)-(X\bullet X^k)\bullet 1&=k\lambda_0\left(\frac{\lambda_0}{k+1}X^{k+1}+\lambda_kX\right)
-\left(\frac{\lambda_0}{k+1}X^{k+1}+\lambda_kX\right)\bullet 1\\
&=\frac{k}{k+1}\lambda_0^2X^{k+1}+\lambda_0\lambda_k X-\lambda_0^2 X^{k+1}-\lambda_0\lambda_kX\\
&=\frac{-1}{k+1} \lambda_0^2X^{k+1};\\ \\
X\bullet (1 \bullet X^k)-(X\bullet 1)\bullet X^k&=0-\lambda_0\left(\frac{\lambda_0}{k+1}X^{k+1}+\lambda_kX\right)\\
&=\frac{-1}{k+1} \lambda_0^2X^{k+1}-\lambda_0\lambda_kX.
\end{align*}
By the preLie relation, $\lambda_0\lambda_k=0$. As $\lambda_0\neq 0$, $\lambda_k=0$.

Finally, $X^k\bullet X^l=\lambda_0\frac{k}{l+1}X^{k+l}$ for all $k,l\geq 0$: this is the preLie product of $\g^{(1)}(1,\lambda_0,1)$.\\

{\it Fourth step.} Second case: $\lambda_0=\lambda_1=0$. Let us prove that $\lambda_k=0$ for all $k\geq 0$. It is obvious if $k=0,1$.
Let us assume that $\lambda_0=\ldots=\lambda_{k-1}=0$, with $k\geq 2$.  Then $X^i\bullet X^j=0$ for all $j<k$, $i\geq 0$. Hence:
$$X\bullet(X^{k+1}\bullet X^{k-1})=(X\bullet X^{k+1})\bullet X^{k-1}=(X\bullet X^{k-1})\bullet X^{k+1}=0.$$
By the preLie relation, $X\bullet (X^{k-1}\bullet X^{k+1})=0$. Moreover:
\begin{align*}
X \bullet (X^{k-1}\bullet X^{k+1})&=X\bullet\left(\sum_{j=1}^{k-2}\binom{k+1}{k+2-j}(k-1) \frac{\lambda_{k+2-j}}{j}X^{k+2-j}\right)\\
&=X\bullet\left((k-1)\lambda_{k+1}X^{k-1}+(k+1)(k-1)\frac{\lambda_k}{2}X^k\right)\\
&=0+\frac{(k-1)(k+1)}{2}\lambda_k X\bullet X^k\\
&=\frac{(k-1)(k+1)}{2}\lambda_k\left(\sum_{j=1}^{k-1} \binom{k}{k+1-j} \frac{\lambda_{k+1-j}}{j}X^j \right)\\
&=\frac{(k-1)(k+1)}{2}\lambda_k^2 X+0.
\end{align*}
Hence, $\lambda_k=0$.

We finally obtain by the first step $X^k\bullet X^l=0$ for all $k,l\geq 0$: this is the trivial preLie product of $\g^{(4)}(0)$.\\

{\it Fifth step.} Last case: $\lambda_0=0$, $\lambda_1\neq0$. Let us prove that 
$\lambda_k=\displaystyle \frac{k!}{2^{k-1}}\frac{\lambda_2^{k-1}}{\lambda_1^{k-2}}$ for all $k \geq 1$. It is obvious if $k=1$ or $k=2$. 
Let us assume the result at all rank $<k$, with $k\geq 2$. 
\begin{align*}
\pi((X\bullet X) \bullet X^k)&=\pi(\lambda_1 X \bullet X^k)\\
&=\lambda_1\lambda_kX;\\ \\
\pi(X\bullet (X\bullet X^k))&=\pi\left(\sum_{j=1}^k \binom{k}{k+1-j}\frac{\lambda_{k+1-j}}{j}X\bullet X^j \right)\\
&=\sum_{j=1}^k \binom{k}{k+1-j}\frac{\lambda_{k+1-j}\lambda_j}{j}X\\
&=\left(\lambda_k\lambda_1+\sum_{j=2}^{k-1}\frac{1}{j}\binom{k}{k+1-j} \frac{(k+1-j)!j!}{2^{k-j+j-1}} 
\frac{\lambda_2^{k-j+j-1}} {\lambda_1^{k-j-1+j-2}}+\frac{k}{k}\lambda_1\lambda_k\right)X\\
&=\left(2\lambda_1\lambda_k+\sum_{j=2}^{k-1}\frac{k!}{2^{k-1}}\frac{\lambda_2^{k-1}}{\lambda_1^{k-3}}\right)X\\
&=\left(2\lambda_1\lambda_k+(k-2)\frac{k!}{2^{k-1}}\frac{\lambda_2^{k-1}}{\lambda_1^{k-3}}\right)X; \\ \\
\pi((X\bullet X^k)\bullet X)&=\sum_{j=1}^k \binom{k}{k+1-j}\frac{\lambda_{k+1-j}}{j}\pi(X^j \bullet X)\\
&=\sum_{j=1}^k \binom{k}{k+1-j}\frac{\lambda_{k+1-j}}{j}\pi(j\lambda_1 X^j)\\
&=\lambda_1\lambda_kX+0;\\ \\
\pi(X\bullet(X^k\bullet X))&=k\lambda_1 \pi(X\bullet X^k)\\
&=k\lambda_1 \lambda_kX.
\end{align*}
By the preLie relation:
$$\lambda_1\lambda_k-2\lambda_1\lambda_k-(k-2)\frac{k!}{2^{k-1}}\frac{\lambda_2^{k-1}}{\lambda_1^{k-3}}
=\lambda_1\lambda_k-k\lambda_1\lambda_k,$$
which gives, as $\lambda_1\neq 0$ and $k \geq 3$, $\displaystyle \lambda_k=\frac{k!}{2^{k-1}} \frac{\lambda_2^{k-1}}{\lambda_1^{k-2}}$.
Finally, the first step gives, for all $k,l \geq 0$, with $\lambda=\lambda_1$ and $\displaystyle \mu=\frac{\lambda_2}{2\lambda_1}$:
\begin{align*}
X^k \bullet X^l&=\sum_{j=1}^{k+1} k\binom{l}{l+1-j} \frac{\lambda_{l+1-j}}{j}X^{j+k-1}\\
&=\sum_{j=1}^k k\frac{l!(l+1-j)!}{(l+1-j)! (j-1)!j 2^{l-j}} \frac{\lambda_2^{l-j}}{\lambda_1^{l-1-j}}X^{j+k-1}\\
&=\lambda k l! \sum_{j=1}^k \frac{\mu^{l-j}}{j!} X^{j+k-1}\\
&=\lambda k l! \sum_{i=k}^{k+l-1} \frac{\mu^{k+l-i-1}}{(i-k+1)!} X^i.
\end{align*}
This is the preLie product of $\g'(\lambda,\mu)$. \end{proof}\\

As $\g'(\lambda,\mu)$ is a special case of $S(V,f,\lambda)$, this ends the proof of theorem \ref{24}.

\bibliographystyle{amsplain}
\bibliography{biblio}

\end{document}